\documentclass[a4paper,12pt]{article}
\usepackage{hyperref}
\usepackage{amsmath}
\usepackage{amssymb}
\usepackage{amsthm}
\usepackage{mathrsfs}
\newtheorem{theorem}{Theorem}[section]
\newtheorem{lemma}[theorem]{Lemma}

\newtheorem{corollary}[theorem]{Corollary}

\newcommand{\EE}{\mathbb{E}}
\newcommand{\ZZ}{\mathbb{Z}}
\newcommand{\NN}{\mathbb{N}}
\newcommand{\PP}{\mathbb{P}}
\newcommand{\RR}{\mathbb{R}}
\newcommand{\T}{\mathbb{T}}
\newcommand{\A}{{\cal A}}
\newcommand{\B}{{\cal B}}
\newcommand{\F}{\mathbb{F}}

\newcommand{\hH}{{\cal H}}

\newcommand{\X}{{\cal X}}
\newcommand{\Y}{{\cal Y}}
\newcommand{\Z}{{\cal Z}}

\newcommand{\aS}{{\cal S}}
\newcommand{\tT}{{\cal T}}
\newcommand{\Os}{{\Omega}}
\newcommand{\Int}{{\rm int}}

\begin{document}

\title {Ergodic Description of STIT Tessellations}
\author{Servet Mart{\'\i}nez \and  Werner Nagel}

\maketitle

\begin{abstract}
\noindent Let $(Y_t: t > 0)$ be the STIT tessellation process. We 
show that for all polytopes $W$ with nonempty interior 
and all $a>1$, the renormalized random  
sequence $(a^n Y_{a^n}: n\in \ZZ)$ induced in $W$, is 
a finitary factor 
of a Bernoulli shift. As a corollary we get that the renormalized 
continuous time process $(a^t Y_{a^t}: t\in \RR)$ induced in $W$ is 
a Bernoulli flow.  
\end{abstract}

\bigskip

\section{ Introduction and main results }
\label{Sec1}

\medskip

\subsection{ Introduction }
\label{Sub1.0}
Let $Y=(Y_t: t > 0)$ be the STIT tessellation process, which is a Markov
process taking values in the space of tessellations of the $\ell$-dimensional 
Euclidean space $\RR^\ell$. The process $Y$ is spatially 
stationary (that is its law is invariant under translations of the space)
and on every polytope with nonempty interior $W$ (called a window) 
the induced tessellation process, which is 
denoted $Y\wedge W=(Y_t\wedge W: t> 0)$, is a pure jump process. The 
process $Y$
was firstly constructed in \cite{nw} and in Subsection \ref{Sub1.2} we give 
a brief construction and recall some of its main properties. 

\medskip

Our results are stated in Subsection \ref{Sub1.3}. In 
Theorems \ref{station} and \ref{mix22} we show that if 
$a>1$ then the renormalized process 
$\Z=(\Z_t:=a^t Y_{a^t}: t\in \RR)$ is a stationary (in time) Markov process
and its restriction to a window 
$\Z\wedge W=(\Z_t\wedge W: t\in \RR)$ is
mixing. In Theorem \ref{bernoulli} we give an ergodic description of the 
discrete process  $\Z^d\wedge W=(\Z_n\wedge W=a^n Y_{a^n}\wedge W: n\in 
\ZZ)$ on a window $W$, where $\ZZ$ is the set of integers.
There we show that $\Z^d\wedge W$ it is a finitary factor of a 
(generalized) Bernoulli shift with null anticipating length.
We conclude in Corollaries \ref{orns} and \ref{ornsflow}
that $\Z^d\wedge W$ is isomorphic to a Bernoulli shift of infinite 
entropy and that $\Z\wedge W$ is isomorphic to a Bernoulli flow
of infinite entropy defined on a Lebesgue probability space.

\medskip

The proofs of these results are done in Section \ref{sec2}.
We use strongly the fact that we are restricting the 
renormalized process to a window, indeed our main technical result, 
Lemma \ref{fundadef}, gives the probability that in a nested 
sequence of decreasing windows the tessellation is 
reduced to the boundaries of the windows. 

\medskip

We need some background on Lebesgue probability spaces 
and on some elements on ergodic theory which are respectively 
given in Subsection \ref{sbsb1} and Section \ref{suberg}.

\subsection{ Notation and some measurability facts }
\label{Sub1.0'}

\subsubsection{ Notation and product spaces }
\label{sbsb1}

For a set $\X$ we denote by $\B(\X)$ a $\sigma-$field on $\X$ and
the couple $(\X,\B(\X))$ is called a measurable space.
If $\X'\in \B(\X)$ then we will always endow $\X'$
with the trace (or induced) $\sigma-$field 
$B(\X')=\{B\cap \X': B\in B(\X)\}$. When $\nu$ is a probability measure
on $(\X,\B(\X))$, we will denote by $(\X,\B(\X),\nu)$ the
completed probability space, where completed means that 
we have added to $\B(\X)$
all the negligible sets with respect to $\nu$. We will always
consider completed probability spaces, even if we do not explicit it.
Sometimes the completed $\sigma-$field with respect to $\nu$ 
is denoted by $\B(\X)_\nu$  
but often (as we do here) it is not written to avoid overburden 
notation.

\medskip

Let $(\X_i, \B(\X_i))$, $i\in L$, be a collection
of measurable spaces. The (Cartesian) product space 
$\prod_{i\in L}\X_i$ will be endowed 
with the product $\sigma$-field $\otimes_{i\in L}\B(\X_i)$,
which is the smallest $\sigma-$field containing the family of
cylinders. We recall that a cylinder is a set of the 
form $C=\prod_{i\in J} A_i$ with $A_i\in \B(\X_i)$ and $J$ a finite 
subset of $L$. We call $(\prod_{i\in L}\X_i, \otimes_{i\in 
L}\B(\X_i))$ the product measurable space.

\medskip

Let $(\X_i,\B(\X_i),\nu_i)$, $i\in L$, be a  
family of probability spaces. The product measure 
$\otimes_{i\in L} \nu_i$ is such that on
each cylinder $C=\prod_{i\in J} A_i$ it
takes the value $(\otimes_{i\in L} \nu_i)(C)=\prod_{i\in J}
\nu_i(A_i)$. We call $(\prod_{i\in L}\X_i, \otimes_{i\in L}\B(\X_i), 
\otimes_{i\in L} \nu_i)$ the product probability space.

\medskip

When $(\X_i,\B(\X_i))=(\X,\B(\X))$ for all $i\in L$, instead of
$\prod_{i\in L}\X_i$ and $\otimes_{i\in L}\B(\X_i)$
we simply put $\X^L$ and $\B(\X)^{\otimes L}$. And if $\nu_i=\nu$
for all $i\in L$, the product probability measure $\otimes_{i\in L} \nu_i$ 
is simply written as $\nu^{\otimes L}$. 

\medskip

Let $(\X,d)$ be a metric space. It is called a  Polish space if 
it is a complete separable metric space. For instance if $(\X,d)$ is
a compact metric space then it is a Polish space. 
Let $(\X_i,d_i)$, $i\in L$, be a family of Polish spaces.
When $L$ is countable, the product space $\X=\prod_{i\in L}\X_i$ 
can be endowed with a metric $d_\X$ such that it is a Polish space
and the topology generated by $d_\X$ is the product topology. 
It suffices to give $d_\X$ when $L=\NN$, 
being $\NN=\{1,2,..\}$ the set of positive integers.
It is easily checked that the metric
$d_\X(x,y)=\sum_{i\in \NN} 2^{-i}\min(d_i(x_i,y_i),1)$ 
for $x=(x_i: i\in \NN)$, $y=(y_i: i\in \NN)\in \X$ does the job. 

\medskip

When $\X$ is a topological space we will reserve the notation
$\B(X)$ to the Borel $\sigma-$field unless the contrary is
explicitly specified.
Let $(\X_i,\B(\X_i))$, $i\in L$, be a family of Polish spaces 
endowed with their Borel $\sigma-$fields.
Consider the product space $\X=\prod_{i\in L}\X_i$.
Then, on the space $\X$ we can consider both the Borel $\sigma-$field 
and the  product $\sigma-$field. When $L$ is countable both 
$\sigma-$fields coincide,
that is $\otimes_{i\in L}\B(\X_i)=\B(\X)$. This
is not the case when $L$ is non-countable, in this case the
$\sigma-$fields are different. In fact the singletons belong to the
Borel $\sigma-$field but not to the product $\sigma-$field. In the
case $L$ is non-countable, we will denote by
${\widehat \B}(\X)$ the product $\sigma-$field to
distinguish it from the Borel $\sigma-$field.

\medskip

Let $\X$ be a topological space. For a set $B\subseteq \X$
we denote by $\Int(B)$ its interior, by $\overline B$ its closure
and by $\partial B={\overline B}\setminus \Int(B)$ its boundary.

\subsubsection{ Measurability facts }
\label{sbsb2}

We recall that a Lebesgue probability space (or a standard probability 
space) is a probability space isomorphic to the unit interval endowed
with a probability measure which is a convex combination of the Lebesgue 
measure and a pure atomic measure ('pure atomic' means that the measure 
is concentrated on points).
Equivalent definitions and 
properties on these spaces can be found in Appendix 1 in \cite{cfs},
Appendix A in \cite{orns2}, Chapter $3$ in \cite{ito} and \cite{delar}. 
In particular in Theorem $2-3$ in \cite{delar} it is
shown that if $(\X,\B(\X))$ is a Polish
space endowed with its Borel $\sigma-$field and $\nu$ 
is a probability  measure on it, then the completed probability space 
$(\X,\B(\X),\nu)$ is Lebesgue.
Hence, if  $\X'\in \B(\X)$ is a Borel set of a Polish space and 
$\nu'$ is a probability measure on $(\X',\B(\X'))$ the complete
probability space  $(\X',\B(\X'),\nu')$ 
is Lebesgue. Let $L$ be a countable set 
and let $(\X_i,\B(\X_i),\nu_i)$, $i\in L$, be a countable family of 
Lebesgue probability spaces, then the product probability space
$(\prod_{i\in L}\X_i, \otimes_{i\in L} \B(\X_i), \otimes_{i\in L} 
\nu_i)$ is also Lebesgue.

\medskip

Let us introduce the Skorohod topology. Let $\RR_+ =[0,\infty )$.
Let $(\X,d)$ be a metric space. 
We denote by $D_{\X}(\RR_+)$ the space of c\`adlag trajectories 
taking values in $\X$ with time in $\RR_+$. We 
recall that c\`adlag means 
that the trajectories are right continuous and have left limits.
The space $D_{\X}(\RR_+)$ is endowed with the Skorohod topology (see
\cite{ek} chapter $3$), which is metrizable 
(see Corollary 5.5 in Chapter $3$ in \cite{ek}).
Let $d_{\rm{Sk}}^{\X}$ be a metric generating the Skorohod topology. 
When $(\X, d)$ is a separable space we get that 
$(D_{\X}(\RR_+), d_{\rm{Sk}}^{\X})$ is also a separable space 
(see Theorem $5.6$ in Chapter $3$ in \cite{ek}). 
We denote by $\B(D_{\X})$ the Borel $\sigma-$field associated to
$(D_{\X}(\RR_+), d_{\rm{Sk}}^{\X})$. 
From Proposition $7.1$ in \cite{ek} we get that the class of
cylinders in $D_{\X}(\RR_+)$ is a semi-algebra generating 
$\B(D_{\X})$. We will also need
the following straightforward extension to processes with time in 
$\RR$. Let $D_{\X}(\RR)$ be the space of c\`adlag 
trajectories with time in $\RR$ taking values in $\X$. The Skorohod
topology, the metric, the associated Borel $\sigma-$field
and all the previous notions are analogously defined. We point out
that the results previously formulated also hold, in particular the 
family of cylinders is a generating semi-algebra.
We continue to denote the metric and the associated Borel $\sigma-$field 
by $d_{\rm{Sk}}^{\X}$ and $\B(D_{\X})$ respectively, because
we want to avoid overburden notation and because there will be no 
confusion from the context.

\subsection{ The space of tessellations }
\label{Sub1.1}

We will consider tessellations on $\RR^\ell$, with $\ell\ge 1$.

\medskip

A polytope is the compact convex hull of  
a finite point set, and we will always assume that it has nonempty 
interior.
A locally finite covering of polytopes is a countable family of 
polytopes whose union is $\RR^\ell$ and all bounded sets can only 
intersect a finite number of them. These polytopes will be called cells.

\medskip

A tessellation $T$ is a locally finite covering of polytopes 
with disjoint interiors.
We denote by $\T$ the space of tessellations of $\RR^\ell$. 
We define the boundary
of a tessellation as the union of the boundaries of its cells. 
That is, for $T\in \T$ we define 
$\partial T:=\bigcup_{C\in T} \partial C$

\medskip

Let $\F$ be the family of closed sets of $\RR^\ell$ endowed with 
the Fell topology $\tT$, for definition and properties see \cite{sw},
Subsections $12.2$ and $12.3$. We denote by
$\F'=\F\setminus \{\emptyset\}$ the class of nonempty closed sets. 
We have that $(\F,\tT)$ is a compact
Hausdorff space with a countable base, so it is metrizable and 
$d$ denotes a metric on $\F$ whose topology is $\tT$. Since 
$(\F,d_\F)$ is a compact metric space, it is a Polish space 
(see Subsection \ref{sbsb1}). 
The set $\F'$ is an open set in $\tT$. Let $\tT'$ be the restriction of
$\tT$ to $\F'$, then $(\F',\tT')$ is a locally compact Hausdorff space 
with a countable base. 

\medskip

Let us denote by $\F(\F')$ the  family of closed sets of $\F'$. 
We endow it with the Fell topology and
denote by $\B(\F(\F'))$ the associated Borel $\sigma-$field. 
Each tessellation $T\in \T$, as a countable collection of polytopes
is an element of $\F(\F')$, so $\T\subset \F(\F')$. Furthermore in 
Lemma 10.1.2. in \cite{sw} it was shown that $\T\in \B(\F(\F'))$. 

\medskip

We will often enumerate the family of countable cells of a tessellation 
$T\in \T$ in a prescribed and measurable form as 
$T=\{C(T)^l:l=1,\ldots \}$. For a tessellation $T$ such that 
the origin $0$ is in the interior of its cell, the first cell
$C(T)^1$ in the enumeration will be the one containing $0$.

\medskip

Let $W\subset \RR^\ell$ be a fixed polytope with nonempty interior, 
we call it a window. As before,
$\F_W$ denotes the set of closed subsets of $W$ and we endow it with
the Fell topology, and we put $\F'_W=\F_W\setminus \{\emptyset\}$.
A tessellation $R$ in $W$ is the collection of all the cells
of a locally finite countable covering of $W$ by polytopes with disjoint 
interiors. We denote by $\T_W$ the space of tessellations of $W$.
Since $W$ is compact the locally finiteness property
implies that every $R\in \T_W$ is constituted by a finite set 
of cells, and we will denote by $|R|$ the number of the cells. 
Each $R\in \T_W$ is an element of $\F(\F'_W)$.
As before we can endow $\F(\F'_W)$
with the Fell topology which is metrizable and we denote by 
$d_{\F_W}$ a metric generating this topology. The space 
$(\F(\F'_W), d_{\F_W})$ is a compact metric space, 
we denote by $\B(\F(\F'_W))$ its Borel $\sigma-$field. We have 
$\T_W\in \B(\F(\F'_W))$, in fact the proof of Lemma 10.1.2. 
in \cite{sw} also works in this case. 
As before we also define the boundary of a tessellation 
$R\in \T_W$ by the union of the boundaries of its cells, 
$\partial R:=\bigcup_{C\in R} \partial C$.
The trivial tessellation $R$ in $\T_W$
has a unique cell which is $R=\{W\}$, and so
its boundary coincides with the boundary of the window 
$\partial R=\partial W$.

\medskip

The tessellations in $\T_W$ can be also seen as induced from a
tessellation in $\T$. In fact each $T\in \T$ induces 
a tessellation $T\wedge W$ in $\T_W$ given by the family of cells 
$T\wedge W=\{C\cap W: C\in T, \ \Int (C\cap W) \neq \emptyset\}$
(note that this set is finite by the locally finiteness property). 
Observe that $T\wedge W=\{W\}$ is 
the trivial tessellation when $W\subseteq C$ for some cell $C\in T$.
When the windows $W, W'$ are such that $W\subseteq W'$, every
$Q\in \T_{W'}$ defines in the same way as before the tessellation 
$Q\wedge W\in \T_W$. In this case $Q\wedge W=\{W\}$
if  $W\subseteq C$ for some cell $C\in Q$.

\medskip

For $a\in \RR$ and $B\subseteq \RR^\ell$, we put $aB=\{ax: x\in B\}$.
Observe that if $W$ is a window and $a\neq 0$ then $aW$ is also a window.
For $T\in \T$ and $a\in \RR\setminus \{0\}$ the
tessellation $aT$ is given by the set of cells
$aT=\{aC: C\in T \}$.
Analogously for a window $W$ and a tessellation
$Q\in \T_W$, the tessellation $aQ\in \T\wedge aW$ is given by  
$aQ=(aC: C\in Q )$. If $W$ is a window containing $0$,
$a>1$, and $Q\in \T_W$, the tessellation $aQ$ belongs to 
$\T_{aW}$ and $W\subset aW$, so we can take the restriction 
$aQ\wedge W\in \T_W$.

\medskip

Since $\F(\F')$ is a compact metric space, for a probability
measure $\nu$ defined on $(\F(\F'),\B(F(\F')))$, the completed 
probability space $((\F(\F'),\B(F(\F')),\nu)$ is  
Lebesgue, see Subsection \ref{sbsb2}. 
Analogously, for any probability measure $\nu_W$ defined
on $(\F(\F'_W),\B(F(\F'_W)))$, the completed probability space 
$((\F(\F'_W),$$\B(F(\F'_W)),\nu_W)$ is Lebesgue. 
Since $\T\in \B(\F(\F'))$ 
its associated Borel $\sigma$-field is $\B(\T)=\{B\cap \T: B\in
\B(\F(\F'))\}$ and for any probability measure $\nu$ defined
on $(\T, \B(\T))$ the completed 
probability space $(\T, \B(\T), \nu)$ is Lebesgue.
Analogously for $\T_W$. We have $\B(\T_W)=\{B\cap \T: B\in
\B(\F(\F'_W))\}$ and for any probability measure $\nu_W$ defined
on $(\T_W, \B(\T_W))$ the completed probability space
$(\T_W, \B(\F_W), \nu_W)$ is Lebesgue.
Also for any countable set $L$ the product probability spaces
$(\T^L, \B(\T))^{\otimes L}, \nu^{\otimes L})$
and $(\T_W^L, (\B(\T_W))^{\otimes L}, \nu_W^{\otimes L})$ 
are Lebesgue.

\subsection{ The STIT tessellation process }
\label{Sub1.2}

Let us construct $Y=(Y_t: t > 0)$ 
a STIT tessellation process (see 
\cite{nw}, \cite{mnw}), which is a Markov processes where each
marginal $Y_t$ takes values in $\T$. 
The law of $Y$ only depends on a (non-zero) 
$\sigma$-finite and translation invariant 
measure $\Lambda$ on the space of hyperplanes $\hH$ in $\RR^\ell$. 
It is assumed that the support set of $\Lambda$ satisfies that 
there is no line in $\RR^\ell$ such that all hyperplanes of the support 
are parallel to it (in order to obtain a.s. bounded cells 
in the constructed tessellation).
For all sets $W\subseteq \RR^\ell$ put 
$$
[W]=\{H\in \hH: H\cap W\neq \emptyset\}\,.
$$ 
The assumptions imply $0<\Lambda([W])<\infty$ for every window $W$. The 
translation invariance of $\Lambda$ yields 
(see e.g. \cite{sw}, Theorem 4.4.1.)
\begin{equation}
\label{homlam}
\Lambda ([cW]) = c\, \Lambda ([W]) \quad \mbox{ for all } c>0.
\end{equation}

Denote by $\Lambda_{[W]}$ the restriction of $\Lambda$ to $[W]$ and by 
$\Lambda_W=\Lambda ([W])^{-1}\Lambda_{[W]}$ the normalized probability 
measure.

\medskip

Let us first construct
$Y\wedge W=(Y_t \wedge W: t\ge 0)$ for a window $W$. We note that 
even if for $t=0$ 
the object $Y_0$ does not exist as a tessellation of the whole $\RR^\ell$ 
we define $Y_0\wedge W= \{W\}$ the trivial tessellation for the 
window $W$. 
Let us take two independent families of 
independent random variables $D=(d_{n,m}: n,m\in \NN)$ and 
$\tau=(\tau_{n,m}: n,m\in \NN)$, where each $d_{n,m}$ has distribution 
$\Lambda_W$ and each $\tau_{n,m}$ is exponentially distributed with 
parameter $1$. We define
a sequence of increasing random times $(S_n: n\ge 0)$ and a 
sequence of random tessellations $(Y_{S_n}\wedge W: n\ge 0)$ with,
$S_0=0$ and $Y_0\wedge W= \{W\}$. The process $Y \wedge W$ 
will satisfy
\begin{equation}
\label{cadlagY}
Y_t \wedge W=Y_{S_n} \wedge W, \;\, t\in [S_n,S_{n+1}). 
\end{equation}
The definition of
$(S_n: n\ge 0)$ and $(Y_{S_n} \wedge W : n\ge 0)$ is done by an inductive 
procedure. 
Let $\{C^1_t,...,C^{l_t}_t\}$ be the cells of $Y_{S_n} \wedge W$, we 
put
$$
S_{n+1}=S_n+\tau(Y_{S_n} \wedge W) \hbox{ where }
\tau(Y_{S_n} \wedge W)=\min\{\tau_{n,l}/\Lambda([C^l_t]): l=1,...,l_t\}.
$$
Let $l_0$ be such that $\tau_{n,l_0}/\Lambda([C_t^{l_0}])=\tau(Y_{S_n} 
\wedge W)$  
(it is a.s. uniquely defined). We denote by $m$  the first index
such that $d_{n+1,m}\in  [C_t^{l_0}]$. The variable $d_{n+1,m}$
is distributed as $\Lambda_{C_t^{l_0}}$. The
tessellation $Y_{S_{n+1}} \wedge W$  is defined as the one whose  
cells are $\{C^l_t: l\neq l_0 \} \cup \{ C'_1, C'_2\}$
where $C'_1, C'_2$ is the partition of $C^{l_0}_t$ by the hyperplane
$d_{n+1,m}$.

\medskip

In particular, since $S_1$ is exponentially distributed with parameter 
$\Lambda ([W])$, 
\begin{equation}
\label{primero22}
\PP(\partial(Y_t \wedge W )\cap \Int W \!=\!\emptyset )\!=\! 
\PP(Y_t \wedge W \!=\!\{W\})\!=\!
\PP(Y_t \wedge W \!=\!Y_0 \wedge W)\!=\!e^{-\Lambda([W])} \,.
\end{equation}

The process $Y\wedge W$ is a Markov process.
Also, this construction 
yields a law that is consistent with respect to $W$, that is
if $W$ and $W'$ are windows and $W\subseteq W'$, 
then  $(Y \wedge W') \wedge W \sim Y \wedge W$, where $\sim$ denotes 
the identity of distributions. 
A proof of consistency  showing the 
existence of the law of the process $Y$ was given in \cite{nw}.

\medskip

Since $\Lambda$ is translation invariant, without loss of generality 
we can always use a window
$W$ with the origin $0$ in its interior and we can also 
assume that $\PP-$a.e. at all times the origin belongs to the interior
of the its cell. This cell is called the $0-$cell.  

\medskip

From (\ref{cadlagY}) it follows that for every window $W$
the process $Y\wedge W$ is a pure jump Markov process with 
c\`adlag trajectories, so its trajectories take values in
the space $D_{\T_W}(\RR_+)$. Recall that 
$D_{\T_W}(\RR_+)$ is endowed with the Skorohod topology generated 
by the metric $d_{\rm{Sk}}^{\T_W}$ (see Subsection \ref{sbsb2}). 
Since $(\T_W, d_{\F_W})$ is a separable space, 
$(D_{\T_W}(\RR_+), d_{\rm{Sk}}^{\T_W})$ 
is also separable. $\B(D_{\T_W})$ denotes 
the Borel $\sigma-$field associated to 
$(D_{\T_W}(\RR_+), d_{\rm{Sk}}^{\T_W})$. 
As before $D_{\T_W}(\RR)$ is 
the space of c\`adlag trajectories taking values in $\T_W$ 
with time in $\RR$. 
The respective metric and Borel $\sigma-$field continue to be 
written by $d_{\rm{Sk}}^{\T_W}$ and $\B(D_{\T_W})$.

\medskip

By technical reasons it is useful to consider the closure 
$\overline{\T_W}$ of $\T_W$ in $\F(\F'_W)$. The space 
$D_{\overline{\T_W}}(\RR_+)$ can be also endowed with the Skorohod 
topology which is generated by a metric $d_{\rm{Sk}}^{\overline{\T_W}}$. 
Since $(\overline{\T_W}, d_{\rm{Sk}}^{\overline{\T_W}})$ 
is a Polish space, from Theorem $5.6$ in Chapter $3$ in \cite{ek} 
we get that $(D_{\overline{\T_W}}(\RR_+), d_{\rm{Sk}}^{\overline{\T_W}})$ 
is also a Polish space. 
$\B(D_{\overline{\T_W}})$ denotes its Borel $\sigma-$field. 
Hence, for a window $W$ we can also consider that the trajectories 
of the Markov process $Y\wedge W$ take values in the Polish 
space $(D_{\overline{\T_W}}(\RR_+), d_{\rm{Sk}}^{\overline{\T_W}})$. 
Also the extension $D_{\overline{\T_W}}(\RR)$ to processes with 
times in $\RR$ is needed, all the previous definitions and 
results hold and we also denote by $\B(D_{\overline{\T_W}})$ 
the associated Borel $\sigma-$field.

\subsubsection{ Independent increments relation}
\label{indincr}
It is useful to supply an independence relation on the
increments of the Markov process $Y$ which is written in terms
of the following operation. For $T\in \T$ and 
${\vec R}=(R^m: m\in \NN)\in \T^\NN$, 
we define the tessellation $T\boxplus {\vec R}$, referred to 
as the iteration of $T$ and ${\vec R}$, by
its set of cells
$$
T\boxplus {\vec R}\!=\!(C(T)^k\!\cap \!C({R}^k)^l: \, 
k\!=\!1,...;\, l\!=\!1,...;\, 
\Int(C(T)^k\!\cap \!C({R}^k)^l)\!\neq \!\emptyset).
$$
So, we restrict ${R}^k$ to the cell $C(T)^k$ and this is done for all
$k=1,\ldots $.
The same definition holds when the tessellation and the sequence of
tessellations are restricted to some window.

\medskip

To state the independence relation of the increments of $Y$, we fix
a copy of the random process $Y$ and let ${\vec Y}'=({Y'}^m: m\in 
\NN)$ be a sequence of independent copies of $Y$, all of them being 
also independent of $Y$. In particular ${Y'}^m\sim Y$.
For a fixed time $s>0$, we set ${\vec Y}'_s=({Y_s'}^m: m\in \NN)$.
Then, from the construction of $Y$ it is straightforward to see that
the following property holds
\begin{equation}
\label{iterate}
Y_{t+s} \sim Y_t\boxplus {\vec Y}'_s \ \mbox{ for all }t,s>0\,.
\end{equation}
This relation was first stated in Lemma $2$ in \cite{nw}. Moreover
the construction done in the proof of this Lemma $2$ also shows the
following relation. Let ${\vec Y}^{'(i)}$, $i=1,\ldots,j$, 
be a sequence of 
$j$ independent copies of ${\vec Y}'$, which are also independent of $Y$.
Then, for all $0<s_1<...<s_j$ and all $t>0$ we have
\begin{equation}
\label{iterate22}
(Y_{t+s_1},...,Y_{t+s_j}) \sim 
(Y_t\boxplus {\vec Y}^{'(1)}_{s_1},...,
(((Y_t\boxplus {\vec Y}^{'(1)}_{s_1})\boxplus....)
\boxplus{\vec Y}^{'(j)}_{s_j-s_{j-1}})). 
\end{equation}

\subsection{ Elements of ergodic theory }
\label{suberg}

A dynamical system $(\Os, \B(\Os),\mu, \psi)$ is such that 
$(\Os, \B(\Os),\mu)$ is a probability space and $\psi:\Os\to \Os$
is a measure-preserving measurable transformation, that is
$\mu(\psi^{-1}(B))=\mu(B)\;$ $\forall \, B\in \B(\Os)$.
When  $(\Os, \B(\Os),\mu, \psi)$ and 
$(\Os', \B(\Os'),\mu', \psi')$
are two dynamical systems, the measurable map 
$\varphi:\Os\to \Os'$ is called a factor map if it satisfies 
$\varphi\circ \psi=\psi'\circ \varphi$ $\; \mu-$a.e. 
and  $\mu(\varphi^{-1}(B'))=\mu'(B')\;$ $\forall \, B'\in \B(\Os')$.
If a factor map $\varphi$ is one-to-one $\mu-$a.e., onto $\mu'-$a.e. 
and $\varphi^{-1}$ is also measurable,  
then it is called an isomorphism and the
dynamical systems $(\Os, \B(\Os),\mu, \psi)$ and 
$(\Os', \B(\Os'),\mu', \psi')$ are called isomorphic.
When $(\Os, \B(\Os),\mu)$ and $(\Os', \B(\Os'),\mu')$ 
are Lebesgue  probability spaces, the
measurability condition on $\varphi^{-1}$ is not 
explicitly needed in the isomorphism requirements
because it is implied by the other ones.

\medskip

The dynamical system $(\Os, \B(\Os),\mu, \psi)$ is ergodic if
$\mu(\psi^{-1}B\Delta B)=0$ implies $\mu(B)\mu(B^c)=0$
(where as usual we set $A\Delta B= (A\setminus B) \cup (B\setminus A)$). It is
mixing if $\lim\limits_{n\to \infty}\mu(\psi^{-n} A\cap B)
=\mu(A)\mu(B)$ for all $A,\,B\in \B(\Os)$. Mixing implies 
ergodicity. To avoid overburden notation the dynamical
system $(\Os, \B(\Os),\mu, \psi)$ is usually denoted
by $(\Os, \mu, \psi)$.

\medskip

Let $(\aS, \B(\aS))$ be a measurable space (i.e. $\aS$ is endowed
with a $\sigma-$field $\B(\aS)$). Let $L=\NN$ or $L=\ZZ$.  
The shift transformation $\sigma_{\aS}:\aS^L\to \aS^L$ defined by
$\sigma_{\aS}(x_n:n\in L)=(x_{n+1}:n\in L)$ is 
a measurable transformation. If the probability measure $\mu$ 
defined on $(\aS^L, \B(\aS^L))$ is preserved by 
$\sigma_{\aS}$ then $(\aS^L, \B(\aS^L), \mu, \sigma_{\aS})$ 
(or simply $(\aS^L, \mu, \sigma_{\aS})$) is a dynamical system 
called a shift system (or simply a shift). When $L=\NN$ it is called a
one-sided shift, and if $L=\ZZ$ then it is called a two-sided shift.  
An example of a two-sided shift is given by a stationary random 
sequence $\Y^d=(\Y_n: n\in \ZZ)$ with state space $\aS$. Indeed, if
$\mu^{\Y^d}$ is the distribution of $\Y^d$ on $\aS^\ZZ$, the stationary 
property of $\Y^d$ means that $\mu^{\Y^d}$ is $\sigma_{\aS}-$invariant
and so $(\aS^L, \mu^{\Y^d}, \sigma_{\aS})$ is a shift system. 

\medskip

Let us recall the Bernoulli property.
Let $(\aS, \B(\aS), \nu_\aS)$ be a probability
space and $L=\NN$ or $L=\ZZ$. Let 
$(\aS^L, \B(\aS^L), \nu_\aS^{\otimes L})$ be the product 
probability space.
The shift action $\sigma_\aS$ preserves the product probability 
measure $\nu_\aS^{\otimes L}$ and
$(\aS^L,\nu_\aS^{\otimes L},\sigma_\aS)$ is called a
Bernoulli shift, it is two-sided when $L=\ZZ$ and one-
sided when $L=\NN$.
In notation of \cite{orns2} Part I, Section 9, 
$(\aS^\ZZ,\nu_\aS^{\otimes \ZZ},\sigma_\aS)$ is called a
generalized two-sided Bernoulli shift (the name generalized 
is because $\aS$ is not necessarily a countable set).
A Bernoulli shift is mixing (so ergodic).
Let us assume that $(\aS, \B(\aS), \nu_\aS)$ is a Lebesgue
probability space. 
Then the entropy $h(\sigma_\aS,\nu_\aS^{\otimes L})$ 
of the Bernoulli shift satisfies
$h(\sigma_\aS,\nu_\aS^{\otimes L})=H(\nu_\aS)$,
where $H(\nu_\aS)=\infty$ if
$\nu_\aS$ has a non-atomic part and
$$
H(\nu_\aS)=-\sum_{A\in \A(\nu_\aS)} \nu_\aS(A)\log(\nu_\aS(A))
$$
when $\nu_\aS$ is purely
atomic and where  $\A(\nu_\aS)$ denotes the set of its atoms
(singletons of positive $\nu_\aS-$measure). The
Ornstein isomorphism theorem
(see \cite{orns10} and \cite{orns11}) states that two-sided 
Bernoulli shifts (defined on Lebesgue probability spaces) having 
the same entropy are isomorphic.

\medskip

Let us introduce what a finitary factor is. If 
$(\aS^\ZZ,\nu_\aS^{\otimes \ZZ},\sigma_\aS)$ 
and $({\aS'}^\ZZ,\nu_{\aS'}^{\otimes \ZZ},\sigma_{\aS'})$ 
are two two-sided Bernoulli shifts, the
measurable map $\varphi:\aS^\ZZ\to {\aS'}^\ZZ$ is a finitary 
factor map if it is a factor map and $\nu_\aS^{\otimes \ZZ}-$a.e.
in $x=(x_n: n\in \ZZ)\in \aS^\ZZ$
the coordinate $(\varphi (x))_n$ only depends on a finite sequence 
of values $(x_m: m\in [n-M(x), n+M'(x)])$. From
$(\varphi(x))_n=(\sigma_{\aS'}^n\circ 
\varphi(x))_0=(\varphi\circ \sigma_{\aS}^n (x))_0$ 
we get that the finitary property can be stated as: 
 $\nu_\aS^{\otimes \ZZ}-$a.e.
in $x=(x_n: n\in \ZZ)\in \aS^\ZZ$ the $0-$th 
coordinate $(\varphi (x))_0$ only depends on a finite sequence
$(x_m: m\in [-M(x), M'(x)])$. 
We call $M(x)$ and $M'(x)$ the memory length and
the anticipation length (for $x$) respectively.
A finitary isomorphism can be defined in an analogous way.
We note that when the state spaces $\aS$ and ${\aS'}$ are finite, 
a finitary factor is a.e. continuous (that is in a set of 
full measure the factor map is continuous when the product spaces
are endowed with the product topologies).
In \cite{ks1} and \cite{ks2} there was introduced a method to 
construct a finitary isomorphism between two Bernoulli shifts of the
same entropy with finite state spaces $\aS$ and ${\aS'}$. 
In \cite{pt} the finitary relation is
studied for topological Markov chains with finite state spaces.

\medskip

A flow (or continuous time dynamical system) 
$(\Os, \B(\Os),\mu, (\psi^t \!:\! t\! \in \! \RR))$ is such that 
$(\Os, \B(\Os),\mu)$ is a probability space and 
$\psi^t:\Os\to \Os$
is a measure-preserving measurable transformation for all $t\in \RR$. 
All the previous notions can be extended from dynamical systems to flows,
in particular ergodicity, mixing and isomorphism of flows.
The shift flows are defined with respect to the shift 
transformations
$\sigma^t(x_s\!:\! s\in \RR)=(x_{s+t}\!:\! s\! \in \! \RR)$ for $t\in \RR$. 
An example of a shift flow is given by a stationary random
process $\Y=(\Y_t: t\in \RR)$ with state space $\aS$.
If $\mu^\Y$ is the distribution of $\Y$ on 
the product measurable space $(\aS^\RR, \widehat{\B}(\aS^\RR))$
then the stationary property of $\Y$ means that
$\mu^\Y$ is $\sigma^t_{\aS}-$invariant
for all $t\in \RR$ and so
$(\aS^\RR, \widehat{\B}(\aS^\RR),\mu^\Y, 
(\sigma^t_{\aS}\!:\! t\! \in \!\RR))$ is a shift flow.
In the case $(\aS,d_\aS)$ is a metric space and the stationary random
process $\Y$ has c\`adlag trajectories, 
let $\mu^\Y$ be the distribution of $\Y$ on $D_{\aS}(\RR)$. 
The stationary property of $\Y$ means that 
$\mu^\Y$ is $\sigma^t_{\aS}-$invariant for $t\in \RR$ and so
$(D_{\aS}(\RR), \B(D_{\aS}),\mu^\Y, (\sigma^t_{\aS}: t\! \in \!\RR))$ is a 
shift flow.

\medskip

A Bernoulli flow is defined in Section $12$, part $2$ in
\cite{orns2} as a flow $(\Os, \B(\Os),\mu, (\psi^t: t\! \in \! \RR))$
such that $(\Os, \B(\Os),\mu, \psi^1)$ is isomorphic 
to a Bernoulli shift. The entropy of the flow is defined to be
the entropy of $(\Os,\mu, \psi^1)$.
The isomorphism theorem for 
Bernoulli flows, Theorem $4$ in Section $12$, part $2$ 
in \cite{orns2}, states that if 
$(\Os, \B(\Os),\mu, (\psi^t\!:\! t\! \in \!\RR))$ and 
$(\Os', \B(\Os'),\mu', ({\psi'}^t\!:\! t\! \in \! \RR))$ are two Bernoulli 
flows 
with the same entropy and such that its completed probability spaces 
$(\Os, \B(\Os),\mu)$ and $(\Os', \B(\Os'),\mu')$ 
are Lebesgue, then the two flows isomorphic.

\subsection{ Renormalized stationary tessellation process and 
main results }
\label{Sub1.3}

Fix $a>1$ and define the process $\Z=(\Z_s: s\in \RR)$ by
$$
\Z_s=a^s Y_{a^s} \,, \;\;\;s\in \RR\,.
$$
Note that $\Z_0=Y_1$. In this context the $t-$shift transformation 
$\sigma_\T^t$ can be expressed as 
$$
\sigma_\T^t \circ \Z=((\sigma^t\circ \Z)_s: s\in \RR) \hbox{ with }
(\sigma_\T^t\circ \Z)_s=\Z_{s+t}.
$$
For any window $W$ we set $\Z \wedge W=(\Z_s\wedge W: s\in \RR)$, and
the shift $\sigma_\T^t$ has an analogous expression as above.  

\medskip

We now state our main results whose proofs will be given in the next Section
\ref{sec2}.

\medskip

In the following result, 
the trajectories of the process $\Z$ take values in the product space 
$\T^\RR$, which is endowed
with the product $\sigma-$field $\widehat \B(\T^\RR)$.
We denote by $\mu^\Z$ the law induced by $\Z$ on 
$(T^\RR,\widehat \B(\T^\RR))$.

\begin{theorem}
\label{station}
$\Z$ is a stationary Markov process, this means that for all $t\in \RR$
the equality in distribution $\Z\sim \sigma_\T^t \circ \Z$ is verified. 

\smallskip

Hence $(T^\RR,\widehat \B(\T^\RR), \mu^\Z, (\sigma^t: t\! \in \!\RR))$
is a shift flow.

\end{theorem}

\medskip

When $W$ is a window the process $\Z \wedge W=(\Z_s\wedge W: s\in \RR)$
inherits from $Y \wedge W$ the property of having c\`adlag trajectories.
Since the trajectories of $Y \wedge W$ take values in $D_{\T_W}(\RR_+)$ 
then the trajectories $\Z \wedge W$ take values in 
$D_{\T_W}(\RR)$. We recall that $(D_{\T_W}(\RR), d_{\rm{Sk}}^{\T_W})$ 
is a separable space. By $\B(D_{\T_W})$ we denoted the 
Borel $\sigma-$field 
associated to $(D_{\T_W}(\RR), d_{\rm{Sk}}^{\T_W})$ 
and by $\mu^\Z_W$ we denote the law induced
by $\Z\wedge W$ on $(D_{\T_W}(\RR),\B(D_{\T_W}))$.

\begin{theorem}
\label{mix22}
For any window $W$, the process $\Z \wedge W$ 
is a stationary Markov process that is mixing in time: 
\begin{equation}
\label{mix0}
\lim_{t\to \infty} 
\PP (\Z \wedge W  \in {\widehat A}, 
\sigma_\T^t\circ \Z \wedge W \in {\widehat B}) =
\PP (\Z \wedge W \in {\widehat A}) \PP (\Z \wedge W \in {\widehat B})
\end{equation}
for all events ${\widehat A},\, {\widehat B}$ in the
Borel $\sigma-$field $\,\B(D_{\T_W})$.

\smallskip

Hence, $(D_{\T_W}(\RR), \B(D_{\T_W}), \mu^\Z_W, 
(\sigma^t_\T: t\! \in \! \RR))$ is
a mixing shift flow.  
\end{theorem}

\smallskip

Let $\Z^d=(\Z_n: n\in \ZZ)$ be the restriction of $\Z$ to integer times 
and let $\mu^{\Z^d}$ be the law of $\Z^d$ on $\T^\ZZ$. 
Theorem \ref{station} implies that $\Z^d$ is stationary in time.
As we pointed out in Section \ref{suberg}, the stationary property can be 
stated by saying that $\mu^{\Z^d}$ is preserved by the shift transformation 
$\sigma_{\T}$.  

\medskip

Let $W$ be a window. The random sequence 
$\Z^d\wedge W=(\Z_n\wedge W: n\in \ZZ)$ 
is also stationary, so the law of 
$\Z^d\wedge W$ on $\T_W^\ZZ$, which is denoted by 
$\mu^{\Z^d}_W$, is $\sigma_{\T_W}-$invariant. 
We will give an ergodic description of the two-sided shift
$(\T_W^\ZZ, \mu_W^{\Z^d}, \sigma_{\T_W})$. 
We recall that Theorem \ref{mix22} states that this shift is mixing, 
so it is ergodic.

\bigskip

Let $\xi$ be the law of $Y_1=Z_0$, so
$\xi(B)=\PP(Y_1\in B)=\PP(Z_0\in B)$ for
$B\in \B(\T)$. Let us denote by $\xi_W$ the law of $Y_1\wedge 
W=Z_0\wedge W$, so
\begin{equation}
\label{defm}
\forall B\in \B(\T_W): \;\; \xi_W(B)=
\PP(Y_1\wedge W\in B)=\PP(Z_0\wedge W\in B).
\end{equation}
In the sequel we fix
\begin{equation}
\label{defnu}
\varrho=\xi_W^{\otimes \NN}. 
\end{equation}

The following ergodic property is verified.

\begin{theorem}
\label{bernoulli}
Let $W$ be a window. Then the shift system
$(\T_W^\ZZ, \mu^{\Z^d}_W, \sigma_{\T_W})$
is a factor of the Bernoulli shift
$((\T_W^\NN)^\ZZ, \varrho^{\otimes\ZZ},
\sigma_{\T_W^\NN})$, that is 
$\exists \, \varphi:(\T_W^\NN)^\ZZ\to \T_W^\ZZ$ 
(a factor map) measurable and defined $\varrho^{\otimes \ZZ}-$a.e., which 
satisfies,
\begin{equation}
\label{propf1}
\sigma_{\T_W}\circ \varphi=\varphi\circ \sigma_{\T_W^\NN} 
\;\;\;
\varrho^{\otimes \ZZ}-\hbox{a.e.},
\end{equation}
and
\begin{equation}
\label{propf2}
\varrho^{\otimes \ZZ}\circ \varphi^{-1}=\mu^{\Z^d}_W\,.
\end{equation}
Moreover, the factor map is finitary with null anticipation, 
that is for all $m\in \ZZ$, $\varrho^{\otimes\ZZ}-$a.e. 
in ${\bf R}=({\vec R}_n: n\in \ZZ)\in (\T_W^\NN)^\ZZ$, 
the coordinate $\varphi({\bf R})_m$ of the image point depends only 
on the finite set of coordinates $({\vec R}_n: n\in [-N,m])$ of the point
${\bf R}$. (The memory length $N$ depends on ${\bf R}$). 
\end{theorem}

A consequence is the following result.

\begin{corollary}
\label{orns}
Let $W$ be a window. Then
$(\T_W^\ZZ, \mu^{\Z^d}_W, \sigma_{\T_W})$
is isomorphic to a Bernoulli shift of infinite entropy.
\end{corollary}

Let us give the steps for its proof. In Subsection \ref{sub333} we 
will show that the shift $(\T_W^\ZZ, \mu^{\Z^d}_W, \sigma_{\T_W})$ has 
infinite entropy. 
On the other hand $((\T_W^\NN)^\ZZ, 
\B((\T_W^\NN)^\ZZ),\varrho^{\otimes\ZZ})$
is a Lebesgue probability space. Then the proof that 
$(\T_W^\ZZ, \mu^{\Z^d}_W, \sigma_{\T_W})$ is isomorphic to a 
Bernoulli shift follows from Theorem $6$, page $54$ in \cite{orns2} 
(see also \cite{orns1}) because there it was shown that
a factor of a Bernoulli shift defined on a Lebesgue probability space 
is isomorphic to a Bernoulli shift.

\begin{corollary}
\label{ornsflow}
Let $W$ be a window. Then $\Z\wedge W$ is a Bernoulli flow 
of infinite entropy that
is isomorphic to any other Bernoulli flow of infinite entropy 
defined on a Lebesgue probability space.
\end{corollary}

The proof is as follows. We have that 
$(D_{\overline{\T_W}}(\RR), \B(D_{\overline{\T_W}}), \mu^\Z,
(\sigma_\T^t: t\! \in \! \RR))$ is a flow and from Corollary \ref{orns}
it follows that
$(D_{\overline{\T_W}}(\RR), \B(D_{\overline{\T_W}}), \mu^\Z, \sigma_\T^1)$
is a Bernoulli shift of infinite entropy. Since $(D_{\overline{\T_W}}(\RR),
\B(D_{\overline{\T_W}}), \mu^\Z)$ is a Lebesgue probability space, 
Theorem $4$ in Section $12$, part $2$
in \cite{orns2} gives the result.

\medskip

Observe that all the results in relation with
$\Z^d$ are also true when instead of the discrete time process 
$(\Z_n\wedge W: n\in \ZZ)$ we
consider $(\Z_{h n}\wedge W: n\in \ZZ)$ for a fixed positive real $h$.
In fact this last process corresponds to the former one when 
$a^h$ is used instead of $a$.

\section{ Proof of the Main Results}
\label{sec2}

We recall that without loss of generality we can assume that 
window $W$ contains the origin in its interior, $0\in \Int (W)$.
Also we can assume that $0$ belongs to the interior of $0-$cell
during all the tessellation process $Y$.

\subsection{Proof of Theorem \ref{station}}
Let us first note that since the space $(\F,d_\F)$ is a Polish space
and $\T$ is a Borel subset of $\F$, a probability measure on the
product space $(\T^\RR, {\widehat \B}(\T)^{\otimes \RR})$ is defined by
the finite dimensional distributions verifying the consistency
property (see \cite{ne}, Corollary Section III.3).

\medskip

Hence, to show that $\Z=(\Z_s: s\in \RR)$ is stationary it suffices to 
prove that the finite dimensional distributions are stationary. So, we 
must show that
\begin{eqnarray}
&{}&\forall t>0 \,,\; \forall s_1<...<s_n \; 
\forall B_1,...,B_n\in \B(\T): 
\nonumber \\
&{}& \PP(Z_{s_i+t}\in B_i: i=1,...,n)=
\PP(Z_{s_i}\in B_i: i=1,...,n)\,.
\label{nev1}
\end{eqnarray}

Let us do it. Since $(Y_t: t > 0)$ is a Markov process, so is $(\Z_s: s\in \RR)$. 
On the other hand it was shown in Lemma $5$ in \cite{nw} that
\begin{equation}
\label{homothet}
t Y_t \sim Y_1  \ \mbox{ for all }t>0\,,
\end{equation}
and hence all 1-dimensional distributions of
$(\Z_s: s\in \RR)$ are identical.
Therefore the proof of (\ref{nev1}) will be finished once we show 
that the transition probabilities
from $\Z_s$ to $\Z_{s+t}$ depend only on the time difference $t>0$.
Now, from (\ref{homothet}) and (\ref{iterate}) we get that for all
$z\in \T$ and all measurable $B\in \B(\T)$ it is satisfied,
\begin{eqnarray*}
\PP \left( \Z_{s+t} \in B \, | \, \Z_s\!=\!z \right)
&=& \PP\left(a^{s+t} Y_{a^{s+t}}\in B \,|\, 
Y_{a^s}\!=\!a^{-s}z\right)\\
&=& \PP \left(z  \boxplus a^s {\vec Y}'_{a^{s+t}-a^s} \in a^{-t}B 
\right)
\!=\! \PP \left(z \boxplus {\vec Y}'_{a^t-1} \in a^{-t}B \right).
\end{eqnarray*}
So the stationary property holds.

\subsection{Proof of Theorem \ref{mix22}}

The process $\Z\wedge W$ takes values in the separable metric
space $D_{\T_W}(\RR)$ which is endowed with its 
Borel $\sigma-$field $\B(D_{\T_W})$. As we pointed out in Section 
\ref{Sub1.2}, since $\T_W$ is separable then the class of cylinders in 
$D_{\T_W}(\RR)$ is a semi-algebra generating 
$\B(D_{\T_W})$. From the Carath\'eodory theorem on exterior 
measures we get that for all sets $E\in \B(D_{\T_W})$ and all
$\epsilon>0\;$ exists $E'\in \B(D_{\T_W})$
which is a disjoint union of a finite number of cylinders such
that $\PP(E\Delta E')<\epsilon$.
Therefore it suffices to show the stationary and the mixing property
for the cylinders in $D_{\T_W}(\RR)$. 
Hence, the above proof of Theorem \ref{station} also shows that 
$\Z\wedge W$ is a stationary Markov process when considered 
in $D_{\T_W}(\RR)$.

\medskip

Let us prove (\ref{mix0}). From the above discussion it suffices to 
prove it for cylinders all ${\widehat A}$ and ${\widehat B}$ such that
$\PP (\Z\wedge W  \in {\widehat A})>0$ and 
$\PP (\Z\wedge W  \in {\widehat B})>0$. 
So let
$$
{\widehat A}=\{\Z_{s_1}\in A_1,\ldots,\Z_{s_j}\in A_j\} \hbox{ and } 
{\widehat B}=\{\Z_{u_1}\in B_1,\ldots,\Z_{u_l}\in B_l\}
$$ 
be such that $s_1<...<s_j$, $u_1<...<u_l$ in $\RR$,
$A_1,...,A_j, B_1,...,B_l\in \B(\T\land W)$ and
$\PP(\Z_0\wedge W\in A_p)>0$, $\PP(\Z_0\wedge W\in B_q)>0$
for $p=1,...,j$, $q=1,...,l$. Note that by time 
invariance property shown in Theorem \ref{station} and since in 
(\ref{mix0}) time $t\to \infty$, we can assume $s_j=u_1=0$.

\medskip

First, let us show (\ref{mix0}) in the case $j=l=1$. 
It suffices to show 
that for all $ A\,,B\in \B(\T\land W)$ which satisfy 
$\PP (\Z_0 \wedge W \in A)>0$, $\PP (\Z_0 \wedge W \in B)>0$,
it is fulfilled
\begin{equation}
\label{mix}
\lim\limits_{t\to \infty} 
\PP (\Z_0 \wedge W  \in A, \Z_{t} \wedge W \in B) =
\PP (\Z_0 \wedge W \in A) \PP (\Z_0 \wedge W \in B)\,.
\end{equation}

Let us consider the events 
$\{\partial Y_1\cap \, \Int(a^{-t} W)=\emptyset\}$ with $t\in \RR$. 
This family of sets increases with $t$ and when $t\to \infty$ it 
converges to the set
$\bigcup\limits_{m\in \NN} \{\partial Y_1\cap \, \Int(a^{-m} 
W)=\emptyset\} =\{\partial Y_1\cap \{0\}=\emptyset\}$.
So
\begin{equation}
\label{limit1}
\lim_{t\to \infty} \PP(\partial Y_1\cap \, \Int ( a^{-t} W)=\emptyset) 
= \PP(\partial Y_1\cap \{0\}=\emptyset) = 1.
\end{equation}

For $t>0$
\begin{eqnarray}
\nonumber
&{}& \PP (\Z_0 \wedge W \in A , \Z_{t} \wedge W \in B)
= \PP (Y_{1} \wedge W \in A , a^t Y_{a^{t}} \wedge W  \in B)\\  
\nonumber
&{}& \; = \PP (Y_{1} \wedge W \in A ,
\partial Y_{1}\cap \Int (a^{-t} W) \!\neq\!\emptyset,
a^t Y_{a^{t}} \wedge W \in B )\\  
\label{lun1}
&{}& \;\;\;\; +\, \PP ( Y_{1} \wedge W \in A ,
\partial Y_{1}\cap \Int (a^{-t} W) \!=\!\emptyset, 
a^t Y_{a^{t}} \wedge W \in B).
\end{eqnarray}

From (\ref{limit1}) the first item converges to $0$ for $t\to \infty$, 
in fact
\begin{eqnarray}
\nonumber
&{}& \lim\limits_{t\to \infty}\PP ( Y_{1} \wedge W \in A, 
\partial Y_{1}\cap \Int (a^{-t} W) \neq \emptyset, 
a^tY_{a^{t}} \wedge W \in B)\\
\label{lun2}
&{}& \; \le \lim\limits_{t\to \infty}
\PP (\partial Y_{1}\cap \Int (a^{-t} W) \neq\emptyset)=0.
\end{eqnarray}

Let us now turn to the analysis of the second item. The assumption
$\PP (\Z_0 \wedge W \in A)>0$ and (\ref{limit1}) imply  for sufficiently large
$t>0$ that
$\PP ( Y_{1} \wedge W \in A ,
\partial Y_{1}\cap \Int (a^{-t} W) \!=\!\emptyset )>0$. Thus
\begin{eqnarray*}
&{}& \; \;\;  \PP ( Y_{1} \wedge W \in A ,
\partial Y_{1}\cap \Int (a^{-t} W) \!=\!\emptyset, 
a^t Y_{a^{t}} \wedge W \in B)\\
&{}& \; = \PP (  a^t Y_{1+(a^{t}-1)}
\wedge W \in B \; | \; Y_{1} \wedge W \in A ,
\partial(a^t Y_{1})\cap \Int(W) \!=\!\emptyset )\times \\
&{}& \; \;\;\; \times 
\PP (Y_{1} \wedge W \in A , 
\partial Y_{1}\cap \Int(a^{-t} W) \!=\!\emptyset ).
\end{eqnarray*}
Conditioned on $\partial(a^t Y_{1})\cap \Int(W) \!=\!\emptyset$, the 
Markov property 
and the consistency of the construction (described in Subsection \ref{Sub1.2}) 
yield that $ a^t Y_{1+(a^{t}-1)} \wedge W$ 
is distributed as $ a^t Y_{a^{t}-1} \wedge W$, that is
$$
\PP( a^t Y_{1+(a^{t}-1)} \wedge W \!\in \!B \, | \, 
Y_{1} \wedge W \! \in \! A,
\partial(a^t Y_{1})\cap \Int(W) \!=\!\emptyset )
\!=\! \PP(a^t Y_{a^{t}-1} \wedge W \in B) .
$$
Hence 
\begin{eqnarray}
\nonumber
&{}& \; \;\;  \PP ( Y_{1} \wedge W \in A ,
\partial Y_{1}\cap \Int(a^{-t} W) \!=\!\emptyset, 
a^t Y_{a^{t}} \wedge W \in B)\\
\label{lun3}
&{}& \; = \PP(Y_{1} \wedge W \in A, 
\partial Y_{1}\cap \Int (a^{-t} W) \!=\!\emptyset )
\ \PP ( a^t Y_{a^{t}-1} \wedge W \in B ).
\end{eqnarray}

Note that (\ref{limit1}) also yields
\begin{equation}
\label{lun4}
\lim_{t\to \infty} \PP(Y_{1} \!\wedge\! W \!\in \!A,
\partial Y_{1}\cap \Int(a^{-t}W)\!=\!\emptyset ) \!=\!
\PP(Y_{1} \!\wedge W \!\in A) \!=\!\PP(\Z_0 \!\wedge \!W \!\in \!A).
\end{equation}
Therefore, from the relations 
(\ref{lun1}),(\ref{lun2}), (\ref{lun3}) and (\ref{lun4}), we get 
that the result will be proven once we show
\begin{equation}
\label{mar40}
\lim\limits_{t\to \infty}
\PP (a^t Y_{a^t-1}\wedge W \in B)=\PP(Y_1\wedge W \in B).
\end{equation}
From (\ref{homothet}) we have $\PP (a^t Y_{a^t-1}\wedge W \in B)
=\PP(Y_{1-a^{-t}} \wedge W \in B )$ and so it suffices to show,
\begin{equation}
\label{lun5}
\lim\limits_{t\to \infty} \PP(Y_{1-a^{-t}} \wedge W \in B )=
\PP(Y_1\wedge W \in B).
\end{equation}

For $k\in \NN$ and $t>0$ define the events
\begin{equation}
\label{defdd}
D_{k,t}=\{\partial Y_{a^{-t}}^{'m} \cap \, \Int (W) =\emptyset \;\, 
\forall \, m\in \{ 1,..., k\}\} 
\end{equation}
with ${\vec Y}'$ as introduced in Subsection \ref{indincr}.
Notice that for any fixed $k$ the events are monotonically increasing 
in $t$ because due to the construction of the process the sets 
$Y_{a^{-t}}^{'m}$ are 
decreasing in $t$. Moreover, from (\ref{limit1}) we get
\begin{equation}
\label{limB}
\lim_{t\to \infty } \PP (D_{k,t})= \lim_{t\to \infty} 
\PP(\partial Y_1\cap \, \Int(a^{-t} W)=\emptyset)^k =1 .
\end{equation}

Further, recall that $|Y\cap W|$ denotes the number of cells of $Y\cap W$. 
We have the following decomposition,
\begin{eqnarray}
\nonumber
\PP(Y_{1-a^{-t}} \wedge \! W \!\in \!B )
\!&\!=\sum\limits_{k\in \NN} 
\PP (Y_{1-a^{-t}}\! \wedge \! W \!\in\! B, 
|Y_{1-a^{-t}}\! \cap W|\!=\!k, D_{k,t}^c ) \\
\label{mar41}
\!& \!+\sum\limits_{k\in \NN} \PP (Y_{1-a^{-t}}\!\wedge\!W \!\in\! B, 
|Y_{1-a^{-t}}\!\cap W|\!=\!k, D_{k,t} ). 
\end{eqnarray}

Let us analyze the first sum in (\ref{mar41}). From (\ref{homlam}) we get
$$
\PP(D_{k,t})=e^{-a^{-t}\Lambda ([W])k}.
$$
Then, by independence between $Y$ and $\{D_{k,t}: k\in \NN\}$ and 
by using that $|Y_{s}\cap W|$ increases 
with $s$ we obtain,
\begin{eqnarray*}
&{}&  \sum\limits_{k\in \NN}
\PP(|Y_{1-a^{-t}}\cap W|=k, D_{k,t}^c )= \sum\limits_{k\in \NN}
\PP (|Y_{1-a^{-t}}\cap W|=k) \, \PP (D_{k,t}^c ) \\
&{}&  \le \sum\limits_{k\in \NN}
\PP (|Y_{1-a^{-t}}\cap W|=k)\, 
\left(1-e^{-a^{-t}\Lambda ([W])k}\right) \\
&{}& = \EE\left(1-e^{-a^{-t}\Lambda([W]) 
|Y_{1-a^{-t}}\cap W|}\right)\le 
\EE\left(1-e^{-a^{-t}\Lambda([W])|Y_{1}\cap W|}\right). 
\\
\end{eqnarray*}
Since the term $\left(1-e^{-a^{-t}\Lambda ([W])\, 
|Y_{1}\cap W|}\right)$ is dominated by $1$
and it decreases with $t$, the Monotone Convergence Theorem gives
$$
\lim_{t\to \infty} \EE\left(1-e^{-a^{-t}\Lambda ([W])\, 
|Y_{1}\cap W|}\right) =\EE \left(\lim_{t\to \infty}
\left(1-e^{-a^{-t}\Lambda([W])\, |Y_{1}\cap W|}\right)\right)\!=\!0 .
$$
We have shown that
\begin{equation}
\label{mar42}
\lim_{t\to \infty} \sum\limits_{k\in \NN}  
\PP(|Y_{1-a^{-t}}\cap W|=k , \ D_{k,t}^c )=0 \,.
\end{equation}
Then,
$$
\lim_{t\to \infty} \sum\limits_{k\in \NN}
\PP (Y_{1-a^{-t}} \wedge W \in B ,\
|Y_{1-a^{-t}}\cap W|=k , \ D_{k,t}^c )=0 \,.
$$

Let us turn to the second term in (\ref{mar41}).
With an appropriate measurable numbering of the cells of $Y_{1-a^{-t}}$, 
we get the inclusion of events
$$
\{|Y_{1-a^{-t}}\cap W|=k \} \cap  D_{k,t} \subseteq 
\{ Y_{1-a^{-t}}\boxplus {\vec Y}'_{a^{-t}} = Y_{1-a^{-t}} \},
$$
and so,
\begin{eqnarray*}
&{}& \{|Y_{1-a^{-t}}\cap W|=k\} \cap D_{k,t}\\
&{}& =\{Y_{1-a^{-t}}\boxplus {\vec Y}'_{a^{-t}} = Y_{1-a^{-t}} \} 
\cap \{ |Y_{1-a^{-t}}\cap W|=k \} \cap D_{k,t} .
\end{eqnarray*}

This yields
\begin{eqnarray*}
&{}& \PP (Y_{1-a^{-t}}
\wedge W \in B ,\ |Y_{1-a^{-t}}\cap W|=k, \ D_{k,t} ) \\
&{}& =\PP ((Y_{1-a^{-t}}\boxplus {\vec Y}'_{a^{-t}}) \wedge W \in B,
|Y_{1-a^{-t}}\cap W|=k, \ D_{k,t} ) .
\end{eqnarray*}
We have
\begin{eqnarray*}
&{}& \PP ((Y_{1-a^{-t}}\boxplus {\vec Y}'_{a^{-t}}) \wedge W \in B ,\
|Y_{1-a^{-t}}\cap W|=k, \ D_{k,t} ) \\
&{}& = \PP ((Y_{1-a^{-t}}\boxplus {\vec Y}'_{a^{-t}}) \wedge W \in B ,\
|Y_{1-a^{-t}}\cap W|=k)\\
&{}& \;\;\; -\, \PP ((Y_{1-a^{-t}}\boxplus {\vec Y}'_{a^{-t}}) 
\wedge W \in B, |Y_{1-a^{-t}}\cap W|=k, D_{k,t}^c) .
\end{eqnarray*}
By summing this equality over $k\in \NN$ and by using that the family of 
events $\left(|Y_{1-a^{-t}}\cap W|=k: k\in \NN \right)$ 
is disjoint and covers the whole space we obtain,
\begin{eqnarray*}
&{}&
\sum\limits_{k\in \NN} \PP ((Y_{1-a^{-t}}\boxplus {\vec Y}'_{a^{-t}}) \wedge W
\in B ,\ |Y_{1-a^{-t}}\cap W|=k, \ D_{k,t} ) \\
&{}&= \PP ((Y_{1-a^{-t}}\boxplus {\vec Y}'_{a^{-t}}) \wedge W \in B) \\
&{}& \;\;\; -\, 
\sum\limits_{k\in \NN} \PP((Y_{1-a^{-t}}\boxplus {\vec Y}'_{a^{-t}})
\wedge W \in B ,\ |Y_{1-a^{-t}}\cap W|=k, \  D_{k,t}^c ) .
\end{eqnarray*}
Since $Y_1 \sim  (Y_{1-a^{-t}}\boxplus \vec{Y}'_{a^{-t}})$, also
\begin{eqnarray*}
&{}& \sum\limits_{k\in \NN} \PP ((Y_{1-a^{-t}}\boxplus
{\vec Y}'_{a^{-t}}) \wedge W \!\in \! B,
|Y_{1-a^{-t}}\cap W|\!=\!k, D_{k,t} )\\ 
&{}& = \PP (Y_1 \wedge W \in B)\\
&{}& \;\;\; -\, \sum\limits_{k\in \NN} 
\PP((Y_{1-a^{-t}}\boxplus {\vec Y}'_{a^{-t}})
\wedge W \!\in \!B ,\ |Y_{1-a^{-t}}\cap W|\!=\!k, D_{k,t}^c ) .
\end{eqnarray*}

From (\ref{mar42}) we get,
$$
\lim_{t\to \infty} \, \sum\limits_{k\in \NN}
\PP ((Y_{1-a^{-t}} \boxplus {\vec Y}'_{a^{-t}}) \wedge W \in B ,\
|Y_{1-a^{-t}}\cap W|=k , \ D_{k,t}^c )=0 \,.
$$

Hence
$$
\lim\limits_{t\to \infty} \sum\limits_{k\in \NN}
\PP ((Y_{1-a^{-t}}\boxplus {\vec Y}'_{a^{-t}}) \wedge W \in B ,
|Y_{1-a^{-t}}\cap W|\!=\! k, D_{k,t} ) \\
\!=\!\PP (Y_1 \wedge W \!\in \! B ).
$$
Thus we have shown
$\lim\limits_{t\to \infty}\PP(Y_{1-a^{-t}} \wedge \! W \!\in \!B )
=\PP(Y_{1} \wedge W \in B)$ and  (\ref{mar40}) is verified.
The proof of Theorem \ref{station} in the case $j=l=1$ is complete. 

\medskip

Let us now show the general case $j>0$, $l>0$, $s_1<...<s_j=0$ and
$u_1=0<...<u_l$ in $\RR$, and $A_1,\ldots,A_j, B_1,\ldots,B_l\in 
\B(\T_W)$. 
Since the proof is entirely similar to the case $j=l=1$ we will only give 
the main steps. We put 
$$
\Z[s_1,\ldots,0]\wedge W\in [A_1,\ldots,A_j]:=
\{\Z_{s_1}\wedge W \in A_1,\ldots,\Z_0\wedge W \in A_j\}.
$$ 
The same notation is used for $Y$. We must prove
\begin{eqnarray*}
&{}&
\lim\limits_{t\to \infty} 
\PP (\Z[s_1,\ldots,0]\wedge W\!\in \! [A_1,\ldots,A_j], \Z[t,\ldots,t+u_l] 
\wedge W \!\in \! [B_1,\ldots,B_l]) \\
&{}&\;
=\PP (\Z[s_1,\ldots,0]\wedge W \!\in \! [A_1,\ldots,A_j]) 
\PP(\Z[0,\ldots,u_l] \wedge W \!\in \! [B_1,\ldots,B_l])\,.
\end{eqnarray*}

Let $t>0$. We have
\begin{eqnarray*}
&{}& \{\Z[s_1,..,0]\wedge W\in [A_1,..,A_j], \Z[t,..,t+u_l] \wedge W
\! \in \! [B_1,..,B_l]\}\\
&{}& = \{Y[a^{s_1},..,1]\wedge W\! \in \! [a^{-s_1} A_1,..,A_j],
a^{t} Y[a^t,..,a^{t+u_l}] \wedge W \! \in \! [B_1,..,a^{-u_l}B_l]\}.
\end{eqnarray*}
 By using (\ref{limit1}) and the same arguments as those we used from 
(\ref{lun1}) to (\ref{lun3}) we get,
\begin{eqnarray*}
&{}& \lim\limits_{t\to \infty}
\PP(\Z[s_1,..,0]\wedge W\in [A_1,..,A_j], \Z[t,..,t+u_l] 
\wedge W \! \in \! [B_1,..,B_l])\\
&{}& \; = \lim\limits_{t\to \infty}
\PP (Y[a^{s_1},\ldots,1]\wedge W\in [a^{-s_1}A_1,\ldots,A_j],
\partial Y_{1}\cap \Int (a^{-t} W)=\emptyset,\\
&{}& \; \quad\quad\quad\;\; 
a^{t} Y[a^t,\ldots,a^{t+u_l}] \wedge W \in [B_1,\ldots,a^{-u_l}B_l])\\
&{}& \; = \PP (Y[a^{s_1},\ldots,1]\wedge W\in 
[a^{-s_1}A_1,\ldots,A_j])\times \\
&{}& \quad\quad \times \lim\limits_{t\to \infty}\PP(a^t
Y[a^t-1,\ldots,a^{t+u_l}-1] \wedge W  \in [B_1,\ldots,a^{-u_l}B_l]).
\end{eqnarray*}
Since (\ref{homothet}) implies 
\begin{eqnarray*}
&{}& 
\PP (a^{t} Y[a^t-1,\ldots,a^{t+u_l}-1] \wedge W  \in 
[B_1,\ldots,a^{-u_l}B_l])\\
&{}& \; =\PP(Y[1-a^{-t},\ldots,a^{u_l}-a^{-t}] \wedge W \in 
[B_1,\ldots,a^{-u_l}B_l]),
\end{eqnarray*}
the result will be proven once we show
\begin{eqnarray}
\nonumber
&{}& \lim\limits_{t\to \infty}
\PP(Y[1-a^{-t},\ldots,a^{u_l}-a^{-t}] \wedge W  \in 
[B_1,\ldots,a^{-u_l}B_l])\\
&{}& \; =\PP(Y[1,\ldots,a^{u_l}] \wedge W  
\in [B_1,\ldots,a^{-u_l}B_l]).
\label{mar40*}
\end{eqnarray}

Now, let ${\vec Y}^{'(i)}$, $i=1,...,l$ be $l$ independent 
copies of ${\vec Y}'$, which are also independent of $Y$. Even if 
$\boxplus$ is not associative for sequences 
of tessellations we use 
$Y_{v_0}\boxplus_{i=2}^{l} {\vec Y}^{'(i)}_{v_i}$ to mean 
$\left(.. (Y_{v_0} \boxplus  {\vec Y}^{'(2)}_{v_2})\boxplus..)
\boxplus {\vec Y}^{'(l)}_{v_l} \right)$. 
From the construction in Lemma $2$ in \cite{nw}, see (\ref{iterate22}), 
we have,
\begin{eqnarray*}
&{}&\PP(Y[1\!-\!a^{-t},a^{u_2}\!-\!a^{-t},\ldots,a^{u_l}\!-\!a^{-t}]
\wedge W  \!\in \![B_1,a^{-u_2}B_2,\ldots,a^{-u_l}B_l],\\ 
&{}& \quad\quad\quad\quad |Y_{1-a^{-t}} \cap W|=k)\\
&{}& \; =\!
\PP(Y_{1\!-\!a^{-t}}\wedge W\!\in \! B_1, 
Y_{1\!-\!a^{-t}}
\boxplus{\vec Y}^{'(2)}_{a^{u_2}\!-\!1}\wedge W\!\in \! a^{-u_2} B_2,
\ldots, \\
&{}& \quad\quad\quad
Y_{1\!-\!a^{-t}}\boxplus_{i=2}^{l} {\vec Y}^{'(i)}_{a^{u_i}\!-\!a^{u_{i-1}}}
\wedge W\!\in \! a^{-u_l}B_l, |Y_{1-a^{-t}}\cap W|\!=\!k)\,.
\end{eqnarray*}
This yields
\begin{eqnarray*}
&{}& \PP(Y[1\!-\!a^{-t},a^{u_2}\!-\!a^{-t},\ldots,a^{u_l}\!-\!a^{-t}]
\wedge W  \!\in \![B_1,a^{-u_2}B_2,\ldots,a^{-u_l}B_l])\\
&{}& \; =\!\sum\limits_{k\in \NN}
\PP(Y_{1\!-\!a^{-t}}\wedge W\!\in \! B_1,
Y_{1\!-\!a^{-t}}
\boxplus{\vec Y}^{'(2)}_{a^{u_2}\!-\!1}\wedge W\!\in \! a^{-u_2} 
B_2,\ldots,\\
&{}& \quad\quad\quad
Y_{1\!-\!a^{-t}}\boxplus_{i=2}^{l} {\vec Y}^{'(i)}_{a^{u_i}\!-\!a^{u_{i-1}}}
\wedge W\!\in \! a^{-u_l}B_l, |Y_{1-a^{-t}}\cap W|\!=\!k)\,.
\end{eqnarray*}

Now we use the definition of $D_{k,t}$ done in
(\ref{defdd}) with ${\vec Y}'= {\vec Y}^{'(1)}$.
From the equality of events
\begin{eqnarray*}
&{}&
(Y_{1\!-\!a^{-t}}\wedge W\!\in \! B_1, Y_{1\!-\!a^{-t}}
\boxplus{\vec Y}^{'(2)}_{a^{u_2}\!-\!1}\wedge W\!\in \! 
a^{-u_2}B_2,\ldots,\\
&{}& \quad\quad\quad
Y_{1\!-\!a^{-t}}\boxplus_{i=2}^{l}
{\vec Y}^{'(i)}_{a^{u_i}\!-\!a^{u_{i-1}}}
\wedge W\!\in \! a^{-u_l}B_l,  |Y_{1-a^{-t}}\cap W|\!=\!k, D_{k,t})\\
&{}& \; =\!
(Y_{1\!-\!a^{-t}}\boxplus {\vec Y}^{'(1)}_{a^{-t}}\wedge W\!\in \! B_1, 
(Y_{1\!-\!a^{-t}}\boxplus {\vec Y}^{'(1)}_{a^{-t}})
\boxplus {\vec Y}^{'(2)}_{a^{u_2}\!-\!1}
\wedge W\!\in \! a^{-u_l}B_2,\ldots,\\
&{}& \quad\quad
(Y_{1\!-\!a^{-t}}\boxplus {\vec Y}^{'(1)}_{a^{-t}})
\boxplus_{i=2}^{l} {\vec Y}^{'(i)}_{a^{u_i}\!-\!a^{u_{i-1}}}
\wedge W\!\in \! a^{-u_l}B_l, |Y_{1-a^{-t}}\cap W|\!=\!k, D_{k,t})
\end{eqnarray*}
and by using twice (\ref{mar42}) we get
\begin{eqnarray*}
&{}& \lim\limits_{t\to \infty}
\PP(Y[1\!-\!a^{-t},a^{u_2}\!-\!a^{-t},\ldots,a^{u_l}\!-\!a^{-t}]
\wedge W  \!\in \![B_1,a^{-u_2}B_2,..,a^{-u_l}B_l])\\
&{}& \; =\!
\lim\limits_{t\to \infty}\sum\limits_{k\in \NN}
\PP(Y_{1\!-\!a^{-t}}\wedge W\!\in \! B_1, 
Y_{1\!-\!a^{-t}}\boxplus {\vec Y}^{'(2)}_{a^{u_2}\!-\!1} 
\wedge W\!\in \! a^{-u_2}B_2,.., \\
&{}& \quad\quad\quad
Y_{1\!-\!a^{-t}}\boxplus_{i=2}^{l}
{\vec Y}^{'(i)}_{a^{u_i}\!-\!a^{u_{i-1}}}
\wedge W\!\in \! a^{-u_l}B_l, |Y_{1-a^{-t}}\cap W|\!=\!k)\\
&{}& \; =\!
\lim\limits_{t\to \infty}\sum\limits_{k\in \NN}
\PP(Y_{1\!-\!a^{-t}}\boxplus
{\vec Y}^{'(1)}_{a^{-t}}\wedge W\!\in \! B_1,
(Y_{1\!-\!a^{-t}}\boxplus {\vec Y}^{'(1)}_{a^{-t}})\boxplus
{\vec Y}^{'(2)}_{a^{u_2}\!-\!1}\wedge W\!\in \! a^{-u_2}B_2,\\
&{}& \quad\quad\quad
..,(Y_{1\!-\!a^{-t}}\boxplus {\vec Y}^{'(1)}_{a^{-t}})\boxplus_{i=2}^{l}
{\vec Y}^{'(i)}_{a^{u_i}\!-\!a^{u_{i-1}}}
\wedge W\!\in \! a^{-u_l}B_l, |Y_{1-a^{-t}}\cap W|\!=\!k)\\
&{}& \; =\!
\lim\limits_{t\to \infty}\PP(Y_{1\!-\!a^{-t}}\boxplus
{\vec Y}^{'(1)}_{a^{-t}}\wedge W\!\in \! B_1, 
(Y_{1\!-\!a^{-t}}\boxplus {\vec Y}^{'(1)}_{a^{-t}})\boxplus 
{\vec Y}^{'(2)}_{a^{u_2}\!-\!1}\wedge W\!\in \! a^{-u_2}B_2,\\
&{}& \quad\quad\quad
..,(Y_{1\!-\!a^{-t}}\boxplus {\vec Y}^{'(1)}_{a^{-t}})\boxplus_{i=2}^{l}
{\vec Y}^{'(i)}_{a^{u_i}\!-\!a^{u_{i-1}}}
\wedge W\!\in \! a^{-u_l}B_l).
\end{eqnarray*}
 Finally from
\begin{eqnarray*}
&{}& 
\PP(Y_{1\!-\!a^{-t}}\boxplus {\vec Y}^{'(1)}_{a^{-t}}
\wedge W\!\in \! B_1, (Y_{1\!-\!a^{-t}}\boxplus {\vec Y}^{'(1)}_{a^{-t}})
\boxplus {\vec Y}^{'(2)}_{a^{u_2}\!-\!1}\wedge W\!\in \! a^{-u_2}B_2,\\
&{}& \quad\quad
\ldots,(Y_{1\!-\!a^{-t}}\boxplus {\vec Y}^{'(1)}_{a^{-t}})\boxplus_{i=2}^{l}
{\vec Y}^{'(i)}_{a^{u_i}\!-\!a^{u_{i-1}}}
\wedge W\!\in \! a^{-u_l}B_l)\\
&{}& =\PP(Y_1\wedge W\!\in \! B_1, Y_{a^{u_2}}\wedge W \!\in \! a^{-u_2}B_2,
Y_{a^{u_l}}\wedge W\!\in \! a^{-u_l}B_l),
\end{eqnarray*}
the relation (\ref{mar40*}) follows. The proof of Theorem \ref{mix22} is 
complete. 
$\Box$

\subsection{Proof of Theorem \ref{bernoulli}}

\label{Sub2.2}
We will show some intermediate results -some of them 
having their own interest-, that will be needed in 
the proof of the Theorem. 

\medskip

As announced we assume that the interior of the window $W$  
contains the origin $0$. 

\medskip

Let $\ZZ_-=\{n\in \ZZ : n\leq 0 \}$. For a measurable space 
$(\aS, \B(\aS))$ we shall use the one sided 
Bernoulli shift $\sigma_\aS:\aS^\NN\to \aS^\NN$, 
$(\sigma_\aS(x))_n=x_{n+1}$ 
$\, \forall \, n\in \NN$, and the inverse shift 
$$
\sigma^-_\aS:\aS^{\ZZ_-}\to \aS^{\ZZ_-}, \;\;\; 
(\sigma^-_\aS(x))_n=x_{n-1} \,\; \forall \; n\in \ZZ_-\,.
$$
We will set $\sigma^{-n}_\aS:=(\sigma^-_\aS)^n$ for $n\in \NN$. 
For a probability measure $\nu$ on $(\aS, \B(\aS))$, 
both one-sided Bernoulli shifts $(\aS^\NN, \nu^{\otimes \NN},
\sigma_\aS)$ and  $(\aS^{\ZZ_-}, \nu^{\otimes^{\ZZ_-}},
\sigma^-_\aS)$ 
are canonically isomorphic. We recall that the 
Bernoulli shifts are mixing, so ergodic.

\medskip

In the sequel we use the notion of a boundary $\partial T$ of a 
tessellation $T$ which was defined in Section \ref{Sub1.1} as the
union of the boundaries of its cells.

\medskip

Observe that from property (\ref{iterate}) and definition of $\Z$
it follows that
$$
\Z_{n+1}\sim a \Z_n \boxplus a^{n+1}{\vec Y}'_{a^{n+1}-a^n}\, .
$$
Since
$a^{n+1}{\vec Y}'_{a^{n+1}-a^n}=\frac{a}{a-1}(a^{n}(a-1)
{\vec Y}'_{a^{n}(a-1)})$ we get from (\ref{homothet}),  
\begin{equation}
\label{recrel}
\Z_{n+1}\sim a \Z_n  \boxplus \frac{a}{a-1}{\vec Y}'_1\,.
\end{equation}
Let $({\vec Y}^{'(i)}_1: i\ge 0)$ be independent copies of
${\vec Y}^{'}_1$. A simple recurrence on (\ref{recrel}) and 
(\ref{iterate22}) give the formula
\begin{equation}
\label{recrege}
(\Z_{n+k}: k\in \NN)\sim \left(a^k \Z_n  \boxplus_{i=1}^{k} 
\frac{a^{k+1-i}}{a-1}{\vec Y}^{'(i)}_1: k\in \NN\right) 
\,.
\end{equation}
We recall $M\boxplus_{i=1}^{k} {\vec M}^{'(i)}$ is an abbreviation 
for $\left( \ldots \left( M \boxplus  {\vec M}^{'(1)} \right) 
\boxplus \ldots \right) \boxplus  {\vec M}^{'(k)}  $, where $M$ is a 
tessellation and ${\vec M}^{'(i)}$  a sequence of tessellations.

\medskip

The following fact will be useful. We recall that $\xi_W$ is 
the distribution of $Y_1\wedge W$, see (\ref{defm}).

\begin{lemma}
\label{fundadef}
Let $W$ be a window containing the origin $0$ in its interior.

\smallskip

\noindent Let ${\vec R}^-=(R^{k}: k\in \ZZ_-)$ be a random sequence 
of independent copies of $Y_1\wedge W$, that is 
${\vec R}^-\sim \xi_W^{\otimes \ZZ_-}$. 
Then, for $a>1$ we have
\begin{equation}
\label{formula2}
\forall \, k\in \ZZ_-:\;\;\;
\PP(\partial R^k\cap \Int(a^k W)=\emptyset)=
\PP(\partial R^0\cap \Int(W)=\emptyset)^{a^{k}} \,,
\end{equation}
and
\begin{equation}
\label{formula22}
\PP(\forall \, k\in \ZZ_-:\, 
\partial R^{k}\cap \Int(a^{k}W)=\emptyset)
=\PP(\partial R^0\cap \Int(W)=\emptyset)^{\frac{a}{a-1}}>0\,.
\end{equation}

Moreover
\begin{equation}
\label{productinf}
\PP({\vec R}^-\!: \exists (n_i\!\ge\! 1: i\!\in \!\NN)\nearrow,
\forall i\!\in \!\NN \, \forall k\!\in \!\ZZ_- : 
\partial R^{-n_i +k}\cap \Int(a^{k}W)\!=\!\emptyset)=1.
\end{equation}
($\nearrow$ means strictly increasing; so the sequence $(n_i)$
satisfies $\lim\limits_{i\to \infty}n_i=\infty$). 
\end{lemma}

\begin{proof}
The consistency of the STIT tessellations and (\ref{primero22}) yield for 
all windows $W'\subseteq W$
$$
 \PP (\partial Y_1 \cap \Int (W') =\emptyset )= e^{-\Lambda([W'])} >0 .
$$
Hence for all $k\in \ZZ_-$ we use (\ref{homlam}) to get,
$$
\PP(\partial R^k\cap \Int(a^kW) =\emptyset )= e^{-\Lambda([a^k W])} = 
e^{-a^k \Lambda([W])} = \PP(\partial R^0\cap \Int(W) =\emptyset )^{a^k}
$$
which shows (\ref{formula2}).

\medskip

Further, by monotonicity 
\begin{eqnarray*}
&{}& \PP\left(\forall k\in \ZZ_-: 
\partial R^k\cap \Int (a^k W)=\emptyset \right)\\
&{}& \;\;\; =\lim_{m \to -\infty }\PP\left(\forall k\in \{m,...,0\}: 
\partial R^k\cap \Int (a^k W)=\emptyset \right)\\
&{}& \;\;\; =\lim_{m \to -\infty } 
\prod_{k=m}^0 \exp(-a^k \Lambda([W])) \! = \! 
\lim_{m \to -\infty }  \exp\left( - \Lambda([W])\sum_{k=m}^0 a^k\right)\\
&{}& \;\;\; = \exp\left(- \Lambda([W])\frac{a}{a-1}\right) \! = \! 
\PP (\partial R^0 \cap \Int (W) =\emptyset )^{\frac{a}{a-1}} > 0.
\end{eqnarray*}
This proves (\ref{formula22}).

\medskip

Let us show (\ref{productinf}). 
Consider the inverse Bernoulli shift 
$\left(\T_W^{\ZZ_-}, \xi_W^{\otimes {\ZZ_-}}, 
\sigma^-_{\T_W}\right)$ 
with $(\sigma^-(R'))^k=R'^{k-1}$ $\forall \, k\in \ZZ_-$.
Define
$$
A^*=\{{\vec R}'\in \T_W^{\ZZ_-}: 
\; \partial R'^k\cap \Int(a^{k}W)=\emptyset
\; \forall k\in \ZZ_- \}\,.
$$
By (\ref{formula22}) we have $\xi_W^{\otimes {\ZZ_-}}(A^*)>0$. Since 
Bernoulli shifts are ergodic the Birkhoff Ergodic Theorem gives
$$
\lim\limits_{N\to \infty} \frac{1}{N}\left(\sum_{k=0}^{N-1} 
{\bf 1}_{A^*}({\sigma^{-k}_{\T_W}}({\vec R}'))\right)=
\xi_W^{\otimes {\ZZ_-}}(A^*)\!>\!0 \;\;\;\; 
\xi_W^{\otimes {\ZZ_-}}-\hbox{a.e.}\,.
$$
Therefore 
$\xi_W^{\otimes {\ZZ_-}}-$a.e. in ${\vec R}'\in \T_W^{\ZZ_-}$ 
there exists a strictly increasing sequence 
$(n_i\ge 1: i\in \NN)$ such that 
$\{ {\sigma^{-n_i}_{\T_W}}({\vec R}')\in A\}$ for all $i\in \NN$. 
This is exactly (\ref{productinf}) because 
the distribution of ${\vec R}^- =(R^{k}: k\in \ZZ_-)$ is 
$\xi_W^{\otimes {\ZZ_-}}$.
\end{proof}

\medskip

We will also use the following elementary result.

\begin{lemma}
\label{elemental1}
Let $W$ be a window containing the origin $0$ in its interior. 
Let $T^0$ and $R^0$ be two tessellations, 
$(\vec Q_n: n\in \NN)$ be a family of sequences of
tessellations (so for each $n\in \NN$, 
$\vec Q_n=(Q^m_n: m\in \NN)\in \T^\NN$ is a sequence of tessellations).
Define the following sequences of tessellations in $\T_W$:
\begin{equation}
\label{eleq1}
\forall n\!\in \!\NN: \, T^{n+1}\!=
\!(aT^n\boxplus \frac{a}{a\!-\!1}\vec Q_{n+1})\!\wedge  \!W, 
\; R^{n+1}\!=\!(aR^n\boxplus \frac{a}{a\!-\!1}\vec Q_{n+1})
\!\wedge  \! W.
\end{equation}
Then
\begin{equation}
\label{eleq2}
T^0\wedge a^{-n}W= R^0\wedge a^{-n}W \Rightarrow  \;\; 
T^{n}\wedge W=R^{n}\wedge  W\,.
\end{equation}
\end{lemma}

\begin{proof}
By iterating (\ref{eleq1}) we find 
$$
T^{n}\wedge W= a^n (T^0\wedge a^{-n}W)  \boxplus
\left(\boxplus_{i=1}^{n} 
\frac{a^{n+1-i}}{a-1}{\vec Q_{i}}\right)\wedge W\,.
$$
and the result follows straightforward.
\end{proof}

\bigskip

\noindent {\it Proof of Theorem \ref{bernoulli}}. 
The last part of the Theorem (the fact that the factor map
satisfies the finitary property) will be
part of the construction of the factor map.

\medskip

We recall the notation in (\ref{defnu}), $\varrho=\xi_W^{\otimes \NN}$.
For the tessellation $T=\{C(T)^l: l=1,...\}\in \T_W$ 
(the number of cells is finite)
we prescribe $C(T)^1$ to be the cell containing the origin $0$.
For $\vec {R}=(R^m: m\in \NN)\in \T_W^\NN$, the set of cells of
the tessellation $T\boxplus \vec {R}\in \T_W$ 
is 
$$
\{C(T)^i\cap C(R_i)^j: j=1,...; i=1,...; \mbox{ with }
\Int(C(T)^i\cap C(R_i)^j)\neq \emptyset \}.
$$
As noted in Subsection \ref{Sub1.1} for $b>1$ and $T\in \T_W$,
$b \,T \wedge W$ is also 
in $\T_W$. When $\vec R=(R^m: m\in \NN)\in \T_W^\NN$  
we put $b\, {\vec R} \wedge W=(b R^m \wedge W: m\in \NN)$.

\medskip

The factor map $\varphi:(\T_W^\NN)^\ZZ\to \T_W^\ZZ$ which
must satisfy (\ref{propf1}) and (\ref{propf2})
is constructed in an iterative way: we will
define a sequence of functions $(\varphi^{N}: N\ge 0)$ and
will show that the function
$\varphi=\lim\limits_{N\to \infty}\varphi^{N}$ is pointwise
$\, \varrho^{\otimes \ZZ}-$a.e. defined and fulfills
the property of being a factor.
Then, we start by defining $\varphi^{N}$.

\medskip

Let ${\bf R}=({\vec R}_n: n\in \ZZ)\in (\T_W^\NN)^\ZZ$, 
so each ${\vec R}_n=(R_n^m: m\in \NN)$ is a sequence of 
tessellations in the window $W$. We must 
define the image point $\varphi^{N}({\bf R})=(\varphi^{N}_n({\bf R}): n\in 
\ZZ)$ in $\T_W^\ZZ$. We fix (recall $N\ge 0$), 
$$
\forall \, n\le -N:\;\; \varphi^{N}_n({\bf R})=\{W\}\,, 
$$
and we define by recurrence,
\begin{equation}
\label{defN}
\forall \, n\ge -N:\;\; 
\varphi^{N}_{n+1}({\bf R})=\left(a \, \varphi^{N}_n({\bf R}) 
\boxplus \frac{a}{a-1}{\vec R}_n \right) \wedge W\,.
\end{equation}
We claim that $\varphi=\lim\limits_{N\to \infty}\varphi^{N}$ is
defined $\, \varrho^{\otimes \ZZ}-$a.e. In fact, from
property (\ref{productinf}) in Lemma \ref{fundadef}, 
applied to the sequences $(R^1_n: n\in \ZZ_- )$, we get that 
$\varrho^{\otimes \ZZ}-$a.e.
there exists a sequence $N_i\ge 1$, $N_i\to \infty$ (depending on
${\bf R}$) such that 
$\partial R^1_{k-N_i}\cap \Int(a^{k}W)=\emptyset$ for all 
$k\in \ZZ_-$. Hence, from Lemma \ref{elemental1} 
we deduce that for all $N_i$
$$
\forall N\ge N_i \;\,  \forall n\ge -N_i 
:\;\;\; \varphi^{N}_n({\bf R})=\varphi^{N_i}_n({\bf R}).
$$
Therefore 
\begin{equation}
\label{fico22}
\forall N\ge N_i \;\,  \forall n\ge -N_i 
:\;\;\; \varphi_n({\bf R})= \varphi^{N}_n({\bf R})=
\varphi^{N_i}_n({\bf R}),
\end{equation}
that is all the components $\varphi_n({\bf R})$ for $n\ge -N_i$ 
are well-defined as $\varphi^{N_i}_n({\bf R})$. 
Since the sequence
$N_i\ge 1$ exists $\, \varrho^{\otimes \ZZ}-$a.e. the claim 
is verified, so $\varphi$ is defined $\, \varrho^{\otimes \ZZ}-$a.e. 

\medskip

From the definition of $\varphi^{N}$ we have 
$$
\sigma_{\T_W}(\varphi^{N+1}({\bf R}))=
(\varphi^{N}(\sigma_{\T_W^\NN}({\bf R})).
$$
Then $\varphi$ satisfies the commuting property (\ref{propf1}).
The equality (\ref{fico22}) also shows that the factor map
$\varphi$ satisfies the finitary property stated in the Theorem.

\medskip

Let us now turn to the proof of relation (\ref{propf2}). We first note
that since $\lim\limits_{N\to \infty}\PP(N_i\le N)=1$ for all $N_i$,
from the above construction we obtain
\begin{equation}
\label{firsteq}
\forall \epsilon>0 \; \forall k\in \ZZ \; \exists N(\epsilon,k):\;\,
\PP(\forall N\ge N(\epsilon,k) \, \forall n\ge k: \,
\varphi_n^N=\varphi_n)\!>\!1\!-\!\epsilon.
\end{equation}

\medskip

We proved in Theorem \ref{mix22} that $\Z\wedge W$ is mixing, 
then it is ergodic.
Since $\PP(\partial \Z_n \cap \Int(W)=\emptyset)>0$, the ergodic 
theorem applied to the ergodic stationary sequence $\Z\wedge W$ gives
$$
\lim\limits_{N\to \infty}\frac{1}{N}\left(\sum_{i=0}^{N-1}
{\bf 1}_{\{\partial \Z_i \cap \Int(W)=\emptyset\}}\right)=
\PP(\partial \Z_n\cap \Int(W)=\emptyset)\,>\,0 \;\;\, \PP-\hbox{a.e.}\,.
$$
Then,
$$
\PP\left(\exists n_k\ge 0 :\, \lim\limits_{k\to \infty}n_k=\infty,
\partial \Z_{n_k}\cap \Int(W)=\emptyset\right)=1\,.
$$
Hence for all $\epsilon>0$ there exists $K(\epsilon)>0$ such that
\begin{equation}
\label{conserg}
\PP\left(\exists n\in \{0,..., K(\epsilon)\} :
\partial \Z_n\cap \Int(W)=\emptyset \right)>1-\epsilon\,.
\end{equation}

Consider $\Z^d_-=(\Z_n: n\le 0)$.
For each $M\ge 0$ we define the random sequence 
$V^{M}=(V_n^{M}: n\in \ZZ)$ taking values in $\T_W^\ZZ$ by:
$$
\forall \, n\le -M: \;\; V_n^{M}=\Z_{n+M}\wedge W\,,
$$
and by recurrence
\begin{equation}
\label{defVM}
\forall \, n\ge -M: \;\; V_{n+1}^{M}=\left(a \, V_n^{M}
\boxplus \frac{a}{a-1}{\vec R}_n \right)\wedge W\,.
\end{equation}
The sequence $V^{M}$ depends on $\Z^d_-\wedge W$ and ${\bf R}$, 
if we need to explicit its dependence on ${\bf R}$ we put 
$V_n^{M}({\bf R})$. We claim that $V^{M}\sim \Z\wedge W$. To 
show it first note that from the definition of $V^M$ and by the 
time-stationarity of $\Z\wedge W$ we have 
\begin{equation}
\label{inter22}
(V_n^{M}: n\le -M)=(\Z_n\wedge W: n\le 0) \sim 
(\Z_n\wedge W: n\le -M).
\end{equation}

Let us now define the shifted sequence 
$U^{M}({\bf R})=\sigma_{\T_W}^{-M} V^{M}({\bf R})$ that satisfies
$$
\forall \, n\in \ZZ: \; U^{M}_n({\bf R})= V^{M}_{n-M}({\bf R})\,.
$$
We have $(U^{M}_n: n\in \ZZ_-)=\Z^d_-\wedge W$ and by
stationarity $U^{M}({\bf R})\sim \Z\wedge W$. From
(\ref{conserg}) we obtain
$$
\forall \, M>0:\;\; 
\PP\left(\exists n\in \{0,..., K(\epsilon)\} :
U^{M}_n({\bf R})=\{W\}\right)>1-\epsilon\,.
$$
This is equivalent to
\begin{equation}
\label{int1}
\PP\left(\exists n\in \{-M,..., K(\epsilon)-M\} : 
V^{M}_n({\bf R})=\{W\}\right)>1-\epsilon\,.
\end{equation}

In analogy to (\ref{recrege}) we obtain for all $M \ge 0$ and $l\ge 0$
\begin{eqnarray*}
&{}& V^M_{-M+l}({\bf R})= 
\left( (a^l \Z_0 \wedge W) \boxplus_{i=1}^l 
\frac{a^{l+1-i}}{a-1} {\vec R}_{-M+i-1}  \right) \wedge W , \\
&{}& 
\varphi^M_{-M+l}({\bf R})= \left( (a^l W \wedge W) \boxplus_{i=1}^l 
\frac{a^{l+1-i}}{a-1} {\vec R}_{-M+i-1}  \right) \wedge W . 
\end{eqnarray*}
Thus, if   $V^M_{-M+l}({\bf R})=\{W\}$ for some $M,\ l\ge 0$ then 
necessarily $a^l \Z_0 \wedge W =\{W\}$ and hence also 
$\varphi^M_{-M+l}({\bf R})= \{W\}$. So, if 
$V^{M}_n({\bf R})=\{W\}$ for some $n\in \{-M,..., K(\epsilon)-M\}$
the iteration relations (\ref{defN}) and (\ref{defVM}) allow to deduce
$\varphi^{M}_n({\bf R})=V^{M}_n({\bf R})$ for all $n\ge K(\epsilon)-M$.
Therefore we find
\begin{equation}
\label{int2}
\forall \, N\ge K(\epsilon):\;\;\; \PP\left(\forall n\ge K(\epsilon)-N: 
V^N_n({\bf R})=\varphi^N_n({\bf R})\right)>1-\epsilon\,.
\end{equation}

We can now state the proof of (\ref{propf2}).
Let us fix $k\in \ZZ$ and $l\ge 0$, it is sufficient to show that
$$
\forall  B_j\!\in \! \B(\T_W): \, \PP(\varphi_{k+j}({\bf R})
\!\in \! B_j: j\!=\!0,...,l)\!=\!
\PP(\Z_{k+j}\wedge W \!\in \!B_j: j\!=\!0,...,l).
$$
Fix $M\ge 0$. Since $V^M\sim \Z \wedge W$ it suffices to prove that for
all $\delta>0$,
$$
|\PP(\varphi_{k+j}({\bf R})\in B_j: j=0,...,l)-
\PP(V^M_{k+j}({\bf R})\in B_j: j=0,...,l)|\le \delta.
$$
Therefore, it suffices to show that for any $\delta>0$ we have,
$$
\PP\left(\exists j\in \{0,\ldots ,l\}:\{\varphi_{k+j}({\bf R})\in B_j\}
\Delta \{V^M_{k+j}({\bf R})\in B_j\}\right)\le \delta.
$$
Hence it suffices to prove that for any $\delta>0$ it is satisfied,
\begin{equation}
\label{coup}
\PP(\exists j\in \{0,...,l\}:\;\, \varphi_{k+j}({\bf R})
\neq V_{k+j}({\bf R}))\le \delta.
\end{equation}
To  this purpose let us take $N(\delta/2,k)$ in
(\ref{firsteq}) and $K(\delta/2)$ in (\ref{int2}), to obtain
$$
\PP(\forall N\!\ge \!\max(N(\delta/2,k),K(\delta/2),-k\!+\! K(\delta/2))\, 
\forall n\!\ge\! k: \, V^N_n\!=\!\varphi_n^N\!=\!\varphi_n)\!>\!1\!-\!\delta.
$$
Then, (\ref{coup}) is verified and the proof of 
Theorem \ref{bernoulli} is complete.
\hfill $\Box$

\subsection{Proof of Corollary \ref{orns}}
\label{sub333}
The only relation left to prove is that  
$h(\sigma_{\T_W},\mu^{\Z^d}_W)=\infty$, where 
$h(\sigma_{\T_W},\mu^{\Z^d}_W)$ denotes the entropy of  
$(\T_W^\ZZ, \mu^{\Z^d}_W, \sigma_{\T_W})$.
Recall that $\xi_W$ is the law of $Y_1\wedge 
W=\Z_0\wedge W$. From the Markov property we have 
$$
h(\sigma_{\T_W},\mu^{\Z^d}_W)=\int_{\T_W} 
H(\kappa_T) \,  d\xi_W(T)\,,
$$
where $\kappa_T$ is the law of $\Z_1\wedge W$ conditioned to 
$\Z_0\wedge W =T$. We have $H(\kappa_T)=\infty$ when $\kappa_T$ is not purely 
atomic and $H(\kappa_T)=-\sum_{a\in \A(\kappa_T)} 
\kappa_T(A)\log(\kappa_T(A))$ if $\kappa_T$ is purely atomic and 
$\A(\kappa_T)$ is the set of its atoms. 
So, it suffices to show that 
$$
\xi_W(T\in \T_W: \kappa_T \hbox{ has a non-atomic part })>0\,.
$$
We will show the stronger property: $\kappa_T$ has a non-atomic part
$\xi_W-$a.e.. 
First note that $\kappa_T$ has an atom at $\{aT\wedge W\}$: 
$\kappa_T(\{aT\wedge W\})>0$. This is a consequence of the following facts: 
if $Y_a\wedge  W=Y_1\wedge W$ then $\Z_1\wedge W=aY_a\wedge W=
aY_1\wedge W=a\Z_0\wedge W$; and the construction of the 
process yields that $\PP (Y_a\wedge W = Y_1 \wedge W )>0$.
Also from the construction of the process $Y$ it follows that 
$\kappa_T(\{aT\wedge W\})<1$.

\medskip

Assume that $\kappa_T$ has an atom $T^0\in \T_W$ 
different from the atom $\{aT\wedge W\}$. From the construction there is an 
hyperface
$r$ such that $aT\cup r\subseteq T^0$ and $r\subset H\in \hH$, that is 
$r$ is a part of an hyperplane $H$. The translation invariance 
and $\sigma-$finiteness of the hyperplane measure $\Lambda$ implies
that $\Lambda_W(\{H\})=0$ for all $H\in \hH$. Consequently, the 
hyperface $r$ in $T^0$ appears in the construction with probability $0$.  
We conclude that $\{aT\wedge W\}$ is the unique atom of $\kappa_T$. 
Since $\kappa_T(\{aT\wedge W\})<1$, $\kappa_T$ has a non-atomic part and so
$H(\kappa_T)=\infty$ for all $T\in \T_W$. We conclude  
$h(\sigma_{\T_W},\mu^{\Z^d}_W)=\infty$. $\Box$

\bigskip

\noindent {\bf Acknowledgments} Both authors are indebted for 
the support of Program Basal CMM from CONICYT and W.N. thanks for
DAAD support.

\noindent SERVET MART\'INEZ

\noindent {\it Departamento Ingenier{\'\i}a Matem\'atica and Centro
Modelamiento Matem\'atico, Universidad de Chile,
UMI 2807 CNRS, Casilla 170-3, Correo 3, Santiago, Chile.}
e-mail: smartine@dim.uchile.cl

\noindent WERNER NAGEL

\noindent {\it Friedrich-Schiller-Universitat Jena, Fakultat fur 
Mathematik und Informatik, D-07737 Jena, Germany.}
e-mail: werner.nagel@uni-jena.de

\end{document}